\documentclass[default]{sn-jnlOCA}

\usepackage{amssymb,amscd,amsfonts,amsmath,graphicx,amsthm}

\jyear{2022}%

\theoremstyle{thmstyleone}%
\newtheorem{theorem}{Theorem}
\newtheorem{proposition}[theorem]{Proposition}
\newtheorem{lemma}[theorem]{Lemma}%

\theoremstyle{thmstyletwo}%
\newtheorem{example}{Example}%
\newtheorem{construction}{Construction}%

\theoremstyle{thmstylethree}%
\newtheorem{definition}{Definition}%

\newtheorem{corollary}[theorem]{Corollary}
\newcommand{\rmv}[1]{}

\begin{document}

\title[OCAs and Upper Bounds on Covering Codes in NRT spaces]{Ordered Covering Arrays and Upper Bounds on Covering Codes in NRT spaces}


\author*[1]{\fnm{Andr\'{e}} \sur{Guerino Castoldi}}\email{andrecastoldi@utfpr.edu.br}

\author[2]{\fnm{Emerson} \sur{L. Monte Carmelo}}\email{elmcarmelo@uem.br}

\author[3]{\newline\fnm{Lucia} \sur{Moura}}\email{lmoura@uottawa.ca}

\author[4]{\fnm{Daniel} \sur{Panario}}\email{daniel@math.carleton.ca}

\author[4]{\fnm{Brett} \sur{Stevens}}\email{brett@math.carleton.ca}

\affil*[1]{\orgdiv{Departamento Acad\^{e}mico de Matem\'{a}tica}, 
\orgname{Universidade Tecnol\'{o}gica Federal do Paran\'{a}}, 
\orgaddress{\street{Via do conhecimento, Km 1}}, 
\city{Pato Branco}, \postcode{85503-390}, 
\state{Paran\'{a}}, \country{Brazil}}

\affil[2]{\orgdiv{Departamento de Matem\'{a}tica}, 
\orgname{Universidade Estadual de Maring\'{a}}, 
\orgaddress{\street{Av. Colombo, 5790}}, 
\city{Maring\'{a}}, \postcode{87020-900}, 
\state{Paran\'{a}}, \country{Brazil}}

\affil[3]{\orgdiv{School of Electrical Engineering and Computer Science}, 
\orgname{University of Ottawa}, \orgaddress{\street{800 King Edward St.}}, 
\city{Ottawa}, \postcode{K1K 6N5}, 
\state{Ontario}, \country{Canada}}

\affil[4]{\orgdiv{School of Mathematics and Statistics}, 
\orgname{Carleton University}, \orgaddress{\street{\hspace{1cm} 1125 Colonel By Drive}}, 
\city{Ottawa}, \postcode{K1S 5B6}, 
\state{Ontario}, \country{Canada}}


\maketitle


\abstract{This work shows several direct and recursive constructions of ordered covering arrays using projection, fusion, column augmentation, derivation, concatenation and cartesian product. Upper bounds on covering codes in NRT spaces are also obtained by improving a general upper bound. 
We explore the connection between ordered covering arrays and covering codes in NRT spaces, which generalize similar results for the Hamming metric. Combining the new upper bounds for covering codes in NRT spaces and ordered covering arrays, we improve upper bounds on covering codes in NRT spaces for larger alphabets.
We give tables comparing the new upper bounds for covering codes to existing ones.}

\maketitle

\section{Introduction}

Covering codes deal with the following question: given a metric space, what is the minimum number of balls of fixed radius necessary to cover the entire space?
Several applications, such as data transmission, cellular telecommunications, decoding of errors, and football pool problem, have motivated the investigation of covering codes in Hamming spaces. Covering codes also have connections with other branches of mathematics and computer science, such as finite fields, linear algebra, graph theory, combinatorial optimization, mathematical programming, and metaheuristic search. We refer the reader to the book by Cohen et al.~\cite{cohen1997covering} for an overview of the topic.

Rosenbloom and Tsfasman \cite{rosenbloom1997codes} introduced a metric on
linear spaces over finite fields, motivated by possible applications
to interference in parallel channels of communication systems. This metric,
implicitly posed by Niederreiter \cite{N}, is currently known as the
Niederreiter-Rosenbloom-Tsfasman (NRT) metric (or RT metric or $\rho$ metric sometimes) and is an example of poset metric \cite{firer}.
Since the NRT metric generalizes the Hamming metric, central concepts on codes in
Hamming spaces have been investigated in NRT spaces, such as perfect
codes, maximum distance separable (MDS) codes, linear codes, weight distribution, packing and covering problems \cite{brualdi1995codes,castoldi2015covering,castoldi2018partial,N,rosenbloom1997codes,yildiz2010covering}.  We recommend the book by Firer et al.~\cite{firer} for further concepts and applications of the NRT metric to coding theory.

Covering codes in the NRT metric have been less studied than their packing code counterparts~\cite{firer}.
Brualdi et al.~\cite[Theorem 2.1]{brualdi1995codes} implicitly investigates covering codes in the metric space endowed with a poset metric, more specifically, when the poset is a chain. The same result is explicitly given by Yildiz et al.~\cite[Theorem 2.3]{yildiz2010covering}.  The general problem of covering codes in the NRT space is proposed by Castoldi and Monte Carmelo in \cite{castoldi2015covering}, which deals mainly with upper bounds, recursive relations and some sharp bounds as well as relations with MDS codes.
More recently, the sphere covering bound in NRT spaces is improved in \cite{castoldi2018partial} under some conditions by generalizing the excess counting method.
In the present work, we explore upper bounds and recursive relations for covering codes in NRT spaces using ordered covering arrays (OCAs).

Orthogonal arrays (OAs) play a central role in combinatorial designs with close connections to coding theory; see the book on orthogonal arrays by Hedayat et al.~\cite{hedayat2012orthogonal}. Covering arrays (CAs), also called $t$-surjective arrays, generalize orthogonal arrays and have been received a lot of attention due to their applications to software testing and interesting connections with other combinatorial designs; see the survey paper by Colbourn \cite{colbourn2004combinatorial}.

Ordered orthogonal arrays (OOAs) are a generalization of orthogonal arrays introduced in 1996 independently by
Lawrence \cite{lawrence1996combinatorial} and Mullen and Schmid \cite{mullen1996equivalence}, motivated by their applications to numerical integration. OOAs are also related to the NRT metric, see \cite{castoldi2017ordered}. A survey of constructions of ordered orthogonal arrays is in \cite[Chapter 3]{krikorian2011combinatorial}; see also \cite[Section VI.59.3]{HandbookColbourn}. For a survey of finite field constructions of OAs, CAs and OOAs, see Moura et al.~\cite[Section 3]{mullen2016survey}.

OCAs have been introduced by Krikorian in her master's thesis~\cite{krikorian2011combinatorial}, generalizing several of the mentioned designs (OAs, CAs and OOAs).
Krikorian \cite{krikorian2011combinatorial} investigates recursive and Roux-type
constructions of OCAs as well as other constructions using the columns of a covering array and discusses an application of OCAs to numerical integration (evaluating multi-dimensional integrals). In the present paper, we also give further results for OCAs.

 In this work,
we extend and build upon our results from the conference paper \cite{castoldi2019cai} in the following directions. We give several recursive constructions of OCAs based on projection, fusion, column augmentation, derivation,  and concatenation.
A new approach for recursive construction of OCAs is presented based on column augmentation by adding a new chain to the NRT poset and adding new rows using a cartesian product of suitable arrays (Theorem~\ref{t1}). As a consequence, we obtain a 
bound for the size of OCAs with strength $t$, alphabet $v$ a prime power and NRT posets with $v+2$ chains of length $t$ (Corollary~\ref{corprimepower}).
New upper bounds on covering codes in NRT spaces are obtained by modifying some of the code's codewords that give the general upper bound (Proposition \ref{triviais}) to reduce
the size of the covering code.
We apply the newly found OCA bounds and covering code bounds to obtain new upper bounds on
covering codes in NRT spaces for larger alphabets.

This work is organized as follows. 
We review the basics of the NRT poset, the NRT metric and covering codes in NRT spaces in Section~\ref{sec1}, and define OAs, CAs, OOAs and OCAs in Section \ref{sec2}. Section \ref{sec4} is devoted to new recursive constructions of OCAs yielding recursive relations for OCA numbers.  We generalize and improve upper bounds on covering codes in Section \ref{sec5}. Finally, in Section \ref{sec6}, constructions of covering codes in NRT spaces are derived from OCAs. Tables \ref{tab3} and \ref{tab2} contrast some upper bounds obtained in this paper.

\section{Preliminaries} \label{sec1-2}


\subsection{The Niederreiter-Rosenbloom-Tsfasman metric and covering codes} \label{sec1}

Any poset induces a metric, according to the seminal paper by Brualdi~\cite{brualdi1995codes}. 
Codes based on various poset metrics are presented in a systematic way in the book by Firer et al.~\cite{firer}, and we follow their notation.
Here we include some basic definitions for the NRT poset and metric; we also introduce covering codes under the NRT metric, 
first studied in~\cite{castoldi2015covering}, which is one of the focal points of this paper.

Let $P$ be a finite partial
ordered set (poset) and denote its partial order relation by
$\preceq$. A poset is a {\it chain} when any two elements
are comparable; a poset is an {\it anti-chain} when no two
distinct elements are comparable. A subset $I$ of $P$ is an {\it ideal} of $P$ if $b\in I$ and $a\preceq b$, implies $a\in I$.
The {\it ideal generated} by a subset $A$ of $P$ is the ideal
of smallest cardinality that contains $A$, denoted by
$\langle A\rangle $. 

An element $a\in I$ is {\it maximal in I}
if $a\preceq b$ implies that $b=a$. Analogously, an element
$a\in I$ is {\it minimal in I} if $b\preceq a$ implies that
$b=a$. A subset $J$ of $P$ is an {\it anti-ideal} of $P$ when it is
the complement of an ideal of $P$. If an ideal $I$ has
$t$ elements, then its corresponding anti-ideal has $n-t$
elements, where $n$ is the number of elements in $P$. 

Given positive integers $m$ and $s$, let $\mathcal{R}[m \cdot s]$ be a
set of $ms$ elements partitioned into $m$ blocks $B_{i}$ having
$s$ elements each, where $B_{i}=\{b_{is+1},\ldots, b_{(i+1)s}\}$
for $i=0,\ldots,m-1$ and the elements of each block are ordered
as $b_{is+1} \preceq b_{is+2} \preceq \cdots \preceq b_{(i+1)s}$.
The set $\mathcal{R}[m \cdot s]$ is a poset consisting of the
union of $m$ disjoint chains, each one having $s$ elements. This poset is known as the \emph{Niederreiter-Rosenbloom-Tsfasman poset
$\mathcal{R}[m \cdot s]$}, or briefly the NRT poset $\mathcal{R}[m \cdot s]$. When
$\mathcal{R}[m \cdot s]=[m \cdot s]:=\{1,\ldots,ms\}$, the NRT poset
$\mathcal{R}[m \cdot s]$ is denoted by NRT poset $[m \cdot s]$ and its
blocks are $B_{i}=\{is+1,\ldots, (i+1)s\}$, for $i=0,\ldots,m-1$.



Given the NRT poset $[m\cdot s]$, the {\it NRT distance} between $x=(x_{1},\ldots,x_{ms})$ and
$y=(y_{1},\ldots,y_{ms})$ in $\mathbb{Z}_{q}^{ms}$ is defined in
\cite{brualdi1995codes} as
\begin{center}
 $d_{\mathcal{R}}(x,y)=$\textbar$\langle supp(x-y)\rangle$\textbar$=$\textbar$\langle \{i: x_{i}\neq y_{i}\} \rangle$\textbar.   
\end{center}
A set $\mathbb{Z}_{q}^{ms}$ endowed with the distance $d_{\mathcal{R}}$ is a
{\it Niederreiter-Rosenbloom-\linebreak Tsfasman space}, or simply, an {\it NRT space}. The notation $d_{\mathcal{R}[m\cdot s]}$ can be used to emphasize the structure of the NRT poset $[m\cdot s]$.

The NRT sphere centered at $x$ of radius $R$, denoted by
$B_{d_{\mathcal{R}}}(x,R)=$ \linebreak $\{y\in \mathbb{Z}_{q}^{ms}: d_{\mathcal{R}}(x,y)\leq R\},$
has cardinality given by the formula
\begin{equation*}
V_{q}^{\mathcal{R}}(m,s,R)=1+\sum_{i=1}^{R}\sum_{j=1}^{\min\{m,i\}}
    q^{i-j}(q-1)^{j}\Omega_{j}(i), 
\end{equation*}
where, for $1 \leq i \leq ms$ and $1 \leq j \leq \min\{m,i\},$ the
parameter $\Omega_{j}(i)$ denotes the number of ideals of the
NRT poset $[m \cdot s]$ whose cardinality is $i$ with exactly $j$
maximal elements. The case $s=1$ corresponds to the Hamming sphere $V_{q}(m,R)=V_{q}^{\mathcal{R}}(m,1,R)$.

\rmv{
As expected, the case $s=1$ corresponds to the classical Hamming
sphere. Indeed, each subset produces an ideal formed by minimal
elements of the anti-chain $[m\cdot 1]$, thus the parameters
$\Omega_{i}(i)={m \choose i}$ and $\Omega_{j}(i)=0$ for $j<i$ yield
\begin{align}
\label{bola}
V_{q}(m,R)=V_{q}^{\mathcal{R}}(m,1,R)=1+\sum_{i=1}^{R}(q-1)^{i}{m \choose i}.
\end{align}

In contrast with the Hamming space, the computation of the sum
in Eq.~\ref{lenrt} is not a feasible procedure for a general
NRT space. In addition to the well studied case $s=1$, it is known that
$V_{q}^{\mathcal{R}}(1,s,R)=q^R$ for a space induced by a chain
$[1\cdot s]$, see \cite[Theorem 2.1]{brualdi1995codes} and
\cite{yildiz2010covering}. Also, it is proved in
\cite[Corollary 1]{castoldi2018partial} that, for $R\leq s$,
$$V_{q}^{\mathcal{R}}(m,s,R)=1+\sum_{i=1}^{R}\sum_{j=1}^{\min\{m,i\}}
q^{i-j}(q-1)^{j}{m\choose j}{i-1\choose j-1}.$$
}

\begin{definition}
Given an NRT poset $[m \cdot s]$, a subset $C$ of
$\mathbb{Z}_{q}^{ms}$ is an $R$-{\it covering} of
the NRT space $\mathbb{Z}_{q}^{ms}$ if
for every $x\in \mathbb{Z}_{q}^{ms}$ there is a codeword
$c\in C$ such that $d_{\mathcal{R}}(x,c)\leq R$, or equivalently,
\[
\bigcup_{c\in C}B_{d_{\mathcal{R}}}(c,R)=\mathbb{Z}_{q}^{ms}.
\]
The number $K_{q}^{\mathcal{R}}(m,s,R)$ is the smallest cardinality
of an $R$-covering of the NRT space $\mathbb{Z}_{q}^{ms}$.
\end{definition}

We refer the reader to \cite{castoldi2015covering} for more details on covering codes in NRT spaces.
In \cite{castoldi2015covering,castoldi2018partial,castoldi2019cai} the notation $K_{q}^{RT}(m,s,R)$ is used instead of $K_{q}^{\mathcal{R}}(m,s,R)$. 
In this work, we choose the 
latter from \cite{firer} to use a more current and simplified notation.

In the particular case of $s=1$, an anti-chain $[m\cdot 1]$ induces covering codes in Hamming spaces and the numbers $K_{q}^{\mathcal{R}}(m,1,R)=K_{q}(m,R)$. Determining these numbers is a challenging problem in combinatorial coding theory \cite{cohen1997covering}.  On the other hand, the numbers $K_{q}^{\mathcal{R}}(1,s,R)$ are induced by a chain $[1\cdot s]$, and were completely evaluated in \cite{brualdi1995codes,yildiz2010covering}. The literature on $K_{q}^{\mathcal{R}}(m,s,R)$ for a general NRT poset remains short \cite{castoldi2015covering,castoldi2018partial,castoldi2019cai}.

One upper bound which we will improve is the general upper bound given below.


\begin{proposition}(\cite[Proposition 6]{castoldi2015covering})\label{triviais}
Let $m$ and $s$ be positive integers. For every $q\geq 2$ and $R$ such that $0<R<ms$,
$$ K_{q}^{\mathcal{R}}(m,s,R) \leq q^{ms-R}.$$
\end{proposition}

\subsection{Ordered covering arrays}\label{sec2}

Ordered orthogonal arrays are related to classical combinatorial designs, namely covering arrays (CAs) and orthogonal arrays (OAs).
We refer the reader to a survey paper by Colbourn~\cite{colbourn2004combinatorial} on CAs , 
a book by Hedayat et al.~\cite{hedayat2012orthogonal} on OAs, and a survey paper by Moura et al.\cite[Section 3]{mullen2016survey} for their relations to ordered orthogonal arrays.

Let $t$,  $v$, $\lambda$, $n$, $N$ be positive integers
 and $N\geq\lambda v^{t}$. Let $A$
be an $N\times n$ array over an alphabet $V$ of size $v$. An
$N\times t$ subarray of $A$ is {\it $\lambda$-covered} if it
has each $t$-tuple over $V$ as a row at least $\lambda$ times.
A set of $t$ columns of $A$ is {\it $\lambda$-covered} if the $N\times t$
subarray of $A$ formed by them is $\lambda$-covered; when $\lambda = 1$, we say the set of columns is covered and often omit $\lambda$ from the notation.

\begin{definition}(CA and OA)
Let $N$, $n$, $v$ and $\lambda$ be positive integers such that $2\leq t \leq n$. A  {\it covering array} $CA_{\lambda}(N;t,n,v)$ is an $N\times n$ array $A$ with entries from a set $V$ of size $v$  such that any $t$-set of columns of $A$ is $\lambda$-covered. The parameter $t$ is the {\it strength} of the covering array. The {\it covering array number}  $CAN_{\lambda}(t,n,v)$ is the smallest positive integer $N$ such that a $CA_{\lambda}(N;t,n,v)$ exists. An {\it orthogonal array}, denoted by $OA_{\lambda}(\lambda v^{t};t,n,v)$ or simply by $OA_{\lambda}(t,n,v)$, has a similar definition with the additional requirement that the $\lambda$-coverage must be ``exact'', and as a consequence it 
is the same as a covering array with $N=\lambda v^t$. 
\end{definition}

Roughly speaking, each column in a CA or OA has the same ``importance'' towards coverage requirements. This happens because the set of columns is implicitly labeled by the anti-chain $\mathcal{R}[n \cdot 1]$. In contrast, the importance of each column in an OCA and OOA depends on the anti-ideals of the NRT poset $\mathcal{R}[m \cdot s]$.  
Ordered covering arrays are precisely defined as follows.

\begin{definition} (OCA and OOA)
Let $t$, $m$, $s$, $v$ and $\lambda$ be positive integers such
that $2\leq t \leq ms$. An {\it ordered covering array}
$OCA_{\lambda}(N;t,m,s,v)$ is an $N\times ms$ array $A$ with
entries from an alphabet $V$ of size $v$, whose columns are
labeled by an NRT poset $\mathcal{R}[m \cdot s]$.
For each anti-ideal $J$ of the NRT poset $\mathcal{R}[m \cdot s]$ with \textbar$J$\textbar$=t$,
the set of columns of $A$ labeled by $J$ is $\lambda$-covered.
The parameter $t$ is the {\it strength} of the ordered covering array.
The {\it ordered covering array number}  $OCAN_{\lambda}(t,m,s,v)$  is the
smallest positive integer $N$ such that there exists an $OCA_{\lambda}(N;t,m,s,v)$.
An {\it ordered orthogonal array} has the extra requirement of ``exact" $\lambda$-coverage and thus it is the same as an ordered covering array with $N=\lambda v^t$, denoted by $OOA_{\lambda}(\lambda v^{t};t, m,s,v)$ or simply by
$OOA_{\lambda}(t,m,s,v)$. When $\lambda=1$ we omit $\lambda$ from the notation.
\end{definition}

Ordered covering arrays were first studied by Krikorian~\cite{krikorian2011combinatorial}.
Ordered covering arrays are special cases of variable strength covering arrays \cite{raaphorst2013variable,raaphorst2018variable}, 
where the sets of columns covered are specified by a general hypergraph.

\begin{example}\label{ocaexample}
An OCA of strength 2 with 5 rows:
\[
OCA(5;2,4,2,2)=
\begin{array}{c}
\begin{array}{cc|cc|cc|cc} \ 1 & 2 & 3 & 4 & 5 & 6 & 7 & 8\hspace{1.5mm}  \ \\
\end{array}\\
\left[
\begin{array}{cc|cc|cc|cc}
0 & 1 & 0 & 1 & 0 & 1 & 0 & 1  \\
1 & 1 & 1 & 0 & 0 & 0 & 0 & 0  \\
0 & 0 & 1 & 1 & 1 & 0 & 1 & 0  \\
1 & 0 & 0 & 0 & 1 & 1 & 0 & 0  \\
0 & 0 & 0 & 0 & 0 & 0 & 1 & 1
\end{array}
\right].
\end{array}
\]
The columns of this array are labeled by $[4\cdot 2]=\{1,\ldots,8\}$
and the blocks of the NRT poset $[4 \cdot 2]$ are $B_{0}=\{1,2\}$,
$B_{1}=\{3,4\}$, $B_{2}=\{5,6\}$ are $B_{3}=\{7,8\}$. There are ten
anti-ideals of size 2, namely,
$$ \{1, 2\}, \{3, 4\}, \{5, 6\}, \{7, 8\}, \{2, 4\},
   \{2, 6\}, \{2, 8\}, \{4, 6\}, \{4, 8\}, \{6, 8\}.$$
The $5\times 2$ subarray constructed from each one of
theses anti-ideals covers all the pairs $(0,0)$, $(0,1)$, $(1,0)$
and $(1,1)$ at least once. The array above is not a $CA(5;2,8,2)$ because 
many pairs of columns are not covered, $\{1,4\}$ for example.
\rmv{On the other hand, when the set of columns is labeled by the anti-chain $[8\cdot 1]$, any two distinct columns give an anti-ideal. In particular, for the anti-chain $\{1,4\}$, the vector $(1,1)$ is not covered by the subarray induced by the columns 1 and 4. Therefore the array above is not a $CA(5;2,8,2)$.}
\end{example}

In an $OCA_{\lambda}(N;t,m,s,v)$ such that
$s>t$, each one of the first $s-t$ elements of a block in the NRT
poset $\mathcal{R}[m \cdot s]$ is not an element of any anti-ideal of size $t$.
 Therefore, a column labeled by one of these elements is not part of any $N\times t$ subarray that must be $\lambda$-covered in an OCA.
 So we assume $s\leq t$ from now on.

Two trivial relationships between
the ordered covering array number and the covering array
number $CAN_{\lambda}(t,n,v)$ are:
\begin{enumerate}
\item[(1)] $\lambda v^{t}\leq OCAN_{\lambda}(t,m,s,v) \leq
CAN_{\lambda}(t,ms,v) $;
\item[(2)] For $1\leq s \leq t\leq m$, $CAN_{\lambda}(t,m,v)\leq
OCAN_{\lambda}(t,m,s,v)$.
\end{enumerate}

We observe that if $N=\lambda v^t$,
an $OCA_{\lambda}(\lambda v^{t};t,m,s,v)$ is an ordered orthogonal
array $OOA_{\lambda}(\lambda v^{t};t,m,s,v)$. When $s=1$, an
$OCA_{\lambda}(N;t,m,1,v)$ is a covering array
$CA_{\lambda}(N;t,m,v)$.

The following OCA number is used throughout the paper. Let $s\geq 3$ and $v$ be a prime power; then according to \cite[Theorem 3]{castoldi2017ordered}, there exists an $OOA(v^s;s,v+1,s,v)$. Therefore, for  $2\leq m\leq v+1$,
\begin{equation} \label{eq1}
OCAN(s,m,s,v)=v^s.
\end{equation}

\section{New recursive constructions for ordered covering arrays}\label{sec4}

In this section, we show new recursive relations for ordered 
coverings arrays. Since the proofs of Propositions~\ref{prop8} and~\ref{fusion} and their consequences have already appeared in the extended abstract \cite{castoldi2019cai} 
we only state these results. The others results in this section are new 
recursive constructions for OCAs.

Proposition~\ref{prop8} gives constructions that show the size of a chain can be extended for $s=t$ and OCAN monotonicity for parameters $s$ and $m$.

\begin{proposition}
\label{prop8}
\cite[Proposition 2]{castoldi2019cai}
Let $N, t,m,s, v$ be positive integers. 
\begin{enumerate}
\item[(1)] The existence of an $OCA_{\lambda}(N;t,m,t-1,v)$
implies the existence of an  $OCA_{\lambda}(N;t,m,t,v)$.
\item[(2)] The existence of an
$OCA_{\lambda}(N;t,m,s+1,v)$ implies the existence of an 
$OCA_{\lambda}(N;t,m,s,v)$.
\item[(3)] The existence of an
$OCA_{\lambda}(N;t,m+1,s,v)$ implies the existence of an 
$OCA_{\lambda}(N;t,m,s,v)$.
\end{enumerate}
\end{proposition}

By Proposition \ref{prop8} items $(1)$ and $(2)$, there exists an  $OCA_{\lambda}(N;t,m,t,v)$ if and only if there exists an $OCA_{\lambda}(N;t,m,t-1,v)$, for $t\geq 2$. By this equivalence  when $t=2$, the right hand side OCA has $s=t-1=1$ which corresponds to a covering array. This also shows us
that the constraint $t>2$ must hold in order to have ordered covering arrays essentially different from covering arrays.

To illustrate how $s=t=2$  reduces to covering arrays, let us label the columns of a $CA_{\lambda}(N;2,m,v)$ by the elements of $[m]=\{1,\ldots,m\}$.
Make these columns of the $CA_{\lambda}(N;2,m,v)$ to correspond to columns of an $OCA_{\lambda}(2,m,2,v)$ for NRT poset
given in Fig.~\ref{fg1}. We use the notation $\overline{a}$ to duplicate $a \in[m]$ in the NRT
poset $\mathcal{R}[m\cdot 2]$ in such a way that $a$ and $\overline{a}$ are not comparable, but the columns of $OCA_{\lambda}(N;2,m,2,v)$ labeled by $a$ and $\overline{a}$ are equal.
\begin{figure}[h]
\begin{center}
\includegraphics{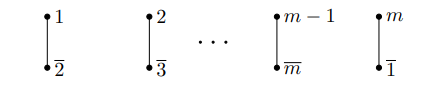}
\caption{Blocks of the NRT poset $\mathcal{R}[m\cdot 2]$.}
\label{fg1}
\end{center}
\end{figure} 

This construction implies 
\begin{equation}\label{thm6}
    OCAN_{\lambda}(2,m,2,v)=CAN_{\lambda}(2,m,v).
\end{equation}

In particular, since $CAN(2,m,2)$ is determined~\cite{kleitman1973families}, we get  
\begin{equation}\label{binCA}
OCAN(2,m,2,2)=\min\left\{N: 
m\leq {N-1 \choose \lfloor \frac{N}{2}\rfloor -1}\right\}.
\end{equation}

The next result shows a relationship for ordered covering arrays numbers over alphabets with different sizes. It is a generalization of
\cite[Lemma 3.1]{colbourn2008strength} and part of \cite[Lemma 3.1]{colbourn2010covering}, and indeed this is a fact that holds for the more general case of variable strength covering arrays~\cite{raaphorst2018variable}.

\begin{proposition}\label{fusion} \cite[Theorem 1]{castoldi2019cai} 
(Fusion) Let $t,m,s, v$ be positive integers. Then
$$OCAN_{\lambda}(t,m,s,v)\leq
OCAN_{\lambda}(t,m,s,v+1)-2.$$
\end{proposition}

Thus for $s\geq 3$, $v$ a prime power, and $2\leq m \leq v+1$, Proposition~\ref{fusion} and Eq.~(\ref{eq1}) give the following upper bound
\begin{equation} \label{corol3} 
OCAN(t,v+1,t,v-1)\leq v^{t}-2.
\end{equation}

We now show new recursive constructions for OCAs.
Chateauneuf and Kreher \cite[Construction D]{Chateauneuf}  develop a form of alphabet augmentation for $s=1$ and strength 3, which yields
\[
CAN(3,n,v)\leq CAN(3,n,v-1)+n\cdot CAN(2,n-1,v-1)+n\cdot (v-1).
\]
Inspired by their construction, we show a form of \emph{alphabet augmentation} for ordered covering arrays for $s=3$ and strength 3.

\begin{theorem}\label{augs2} 
For $m\geq 3$ and $v\geq 3$, $OCAN(3,m,3,v)$ is less than or equal to
$$OCAN(3,m,2,v-1)+mCAN(2,m-1,v-1)+CAN(2,m,v-1)+m(v-1)^2+1.$$ 
\end{theorem}

\begin{proof}
Let $A$ be an $OCA(M;3,m,2,v-1)$, $B$ a $CA(M';2,m-1,v-1)$ and $C$ a $CA(M'';2,m,v-1)$ over the alphabet $\{1,\ldots,v-1\}$. Let $N=M+m M'+M'' +m(v-1)^2+1$. It is sufficient to construct an  $OCA(N;3,m,2,v)$ according to Proposition \ref{prop8} item $(1)$.
For this purpose, the proof is divided into two steps. We first construct an array $\mathcal{B}$ from $B$, an array $\mathcal{C}$ from $C$, and an new array $\mathcal{D}$. Second, we prove that an array $\mathcal{A}$ constructed by vertically juxtaposing the arrays $A$, $\mathcal{B}$, $\mathcal{C}$ and $\mathcal{D}$ is an $OCA(N;3,m,2,v)$.

\emph{Step 1:}  We can construct an $OCA(M';2, m-1,2,v-1)$ from the covering array $B$ as shown in Eq.~(\ref{thm6}).
Let $B^{1},\ldots,B^{m-1}$ be the subarrays of two consecutive columns of $OCA(M';2, m-1,2,v-1)$ such that
$OCA(M';2, m-1,2,v-1)=[B^{1} \ B^{2} \ \ldots \ B^{m-1}]$. We consider the following $(mM')\times 2m$ array $\mathcal{B}$ constructed by inserting  $M'\times 2$ zero arrays in every possible position $l$, $l=1,\ldots,m$, with respect to sequence of arrays  $B^{1},\ldots,B^{m-1}$,  as shown below:
\[
\mathcal{B}=\left[\begin{array}{ccccccc}
                       0 & B^{1} & B^{2} & \cdots & B^{m-3} & B^{m-2} & B^{m-1} \\
                       B^{1} & 0 & B^{2} & \cdots & B^{m-3} & B^{m-2} & B^{m-1} \\
                       \vdots & \vdots & \vdots &  & \vdots & \vdots \\
                       B^{1} & B^{2} & B^{3} & \cdots & B^{m-2} & 0 & B^{m-1} \\
                       B^{1} & B^{2} & B^{3} & \cdots & B^{m-2} & B^{m-1} & 0
                     \end{array}
                     \right].
\]

For $i\in [m]$, let $c_{i}$ be the column $i$ of the covering array $C$. From the covering array $C=[c_{1} \ c_{2} \ \ldots \ c_{m}]$,
we construct the following array:
\[
\mathcal{C}=\left[\begin{array}{cc|cc|cc|c|cc}
                       0 & c_{1} & 0 & c_{2} & 0 & c_{3} & \cdots & 0 & c_{m}
                     \end{array}
                     \right].
\]

Consider the following two arrays $D$ and $E$ of order $(v-1)^2 \times 2$ over the alphabet $\{0,1,\ldots,v-1\}$:
\begin{center}
\begin{minipage}{0.3\textwidth}
\[
D=\left[\begin{array}{cc}
                       0 & 1 \\
                       \vdots & \vdots \\
                       0 & 1 \\
                       \vdots & \vdots \\
                       0 & v-1 \\
                       \vdots & \vdots \\
                       0 & v-1
                     \end{array}
                     \right]
\]
\end{minipage}
\begin{minipage}{0.3\textwidth}
\[
E=\left[\begin{array}{cc}
                       1 & 0 \\
                       \vdots & \vdots \\
                       v-1 & 0 \\
                       \vdots & \vdots \\
                       1 & 0 \\
                       \vdots & \vdots \\
                       v-1 & 0
                     \end{array}
                     \right].
\]
\end{minipage}
\end{center}
Observe that each nonzero element of $\{1,\ldots,v-1\}$ appears exactly $v-1$ times in the nonzero column of $D$ and $E$. We define the array $\mathcal{D}$ of order $(m(v-1)^2+1) \times 2m$ as:
\[
\mathcal{D}=\left[\begin{array}{c|c|c|c|c|c}
                       D & E & E & \cdots & E & E \\
                       E & D & E & \cdots & E & E \\
                       E & E & D & \cdots & E & E \\
                       \vdots & \vdots & \vdots &  & \vdots & \vdots \\
                       E & E & E & \cdots & D & E \\
                       E & E & E & \cdots & E & D \\
                       0 \ 0 & 0 \ 0 & 0 \ 0 & \cdots & 0 \ 0 & 0 \ 0
                     \end{array}
                     \right].
                     \]

\emph{Step 2:} Let $N=M+m M'+M'' +m(v-1)^2+1$. Let $\mathcal{A}$ be the array formed by vertically juxtaposing $A$, $\mathcal{B}$, $\mathcal{C}$, and $\mathcal{D}$. It remains to prove that $\mathcal{A}$ is an $OCA(N;3,m,2,v)$.

Let $B_{i}=\{2i+1,2i+2\}$ be the blocks of the NRT poset $[m\cdot 2]$, for $i=0,\ldots, m-1$. An anti-ideal of size 3 can be formed by choosing the maximal element of three distinct blocks or by choosing the two elements of one block and the maximal element of another block. Let $x$, $y$, and $z$ be nonzero elements of the alphabet $\{0,1,\ldots,v-1\}$.
We can represent the two types of anti-ideals of size 3 of the NRT poset $[m\cdot 2]$ by $x$\textbar $y$\textbar $z$ and $xy$\textbar $z$ (or $x$\textbar $yz$).
The patterns of the $3$-tuples over $\{0,1,\ldots,v-1\}$ considering the two types of the anti-ideals of size 3 and whether 0 is a component of the $3$-tuple are in Table \ref{tab1}.

\begin{table}[h!]
\centering
\caption{Patterns of the 3-tuples for $x,y,z \not=0$}
\label{tab1}
\begin{tabular}{c|cccccc}
  part I & $x$\textbar$y$\textbar $z$ & $xy$\textbar $z$  &  &  & & \\ \hline
  part II & $x$\textbar $y$\textbar 0 & $xy$\textbar 0 & $x$\textbar00 &  & &  \\ \hline
  part III & 0$x$\textbar $y$ &  &  &  & & \\ \hline
  part IV & $x$0\textbar $y$ & $x$\textbar0\textbar0 & $x$0\textbar0 & 0$x$\textbar0 & 0\textbar0\textbar0 & 0\textbar00 \\
\end{tabular}
\end{table}

The patterns in part I are covered by $A$, the patterns in part II are covered by $\mathcal{B}$, the pattern in part III is covered by $\mathcal{C}$,
and the patterns in part IV are covered by $\mathcal{D}$.
Therefore, the array $\mathcal{A}$ is an $OCA(N;3,m,s,v)$, where  $N=M+m M'+M'' +m(v-1)^2+1$.
\end{proof}

We generalize below the derived array bound:
\[
CAN(t-1,n,v)\leq \frac{CAN(t,n+1,v)}{v}
\]
established in \cite{Chateauneuf}.
Indeed,  we construct an $OCA(M;t,m,s,v)$ such that $M\leq \lfloor \frac{N}{v}\rfloor$ by  selecting  rows of an $OCA(N;t+1,m+1,s,v)$ and deleting $s$ columns  labeled by one block of the NRT poset $[(m+1)\cdot s]$. We call this process a \emph{derivation on the number of blocks}.

\begin{proposition}\label{derionm}
$OCAN(t,m,s,v)\leq \left\lfloor \dfrac{OCAN(t+1,m+1,s,v)}{v} \right\rfloor$.
\end{proposition}

\begin{proof}
Let $A$ be an $OCA(N;t+1,m+1,s,v)$.
For a fixed $\alpha \in \{0,1,\ldots,v-1\}$, let $A_{\alpha}$ be the array obtained by choosing the rows of $A$ such that each entry in the last column is equal to $\alpha$, and deleting the last $s$ columns. We choose $\alpha \in \{0,1,\ldots,v-1\}$ that occurs the least number of times in the last column of $A$, to ensure $A_{\alpha}$ has at most  $\lfloor \frac{N}{v} \rfloor$ rows. We claim that the array $A_{\alpha}$ is an $OCA(M;t,m,s,v)$ such that $M\leq \lfloor \frac{N}{v}\rfloor$.
Indeed, let $J$ be an anti-ideal of the NRT poset $[m \cdot s]$ of size $t$. The set $J'=J\cup \{(m+1)s\}$ is an anti-ideal of the NRT poset $[(m+1)\cdot s]$ of size $t+1$. The columns of $A$ labeled by $J'$ are covered. Now, looking at the rows such that the last entry is $\alpha$ in $A$, we have that the columns labeled by $J$ cover all the $t$-tuples over $\{0,1,\ldots,v-1\}$ at least once. Therefore,  $A_{\alpha}$ is an $OCA(M;t,m,s,v)$ such that $M\leq \lfloor \frac{N}{v}\rfloor$.
\end{proof}

In contrast to Proposition \ref{derionm}, we construct an $OCA(M;t,m,s,v)$ such that $M\leq \lfloor \frac{N}{v}\rfloor$ by  selecting  rows of $OCA(N;t+1,m,s+1,v)$ and by choosing  one column in each block of the NRT poset $[m\cdot (s+1)]$  to be deleted. We call this process a \emph{derivation on the size of the blocks}, which is a specific property of ordered covering arrays.

\begin{proposition}
$OCAN(t,m,s,v)\leq \left\lfloor \dfrac{OCAN(t+1,m,s+1,v)}{v} \right\rfloor$.
\end{proposition}

\begin{proof}
Let $I=\cup_{i=0}^{m-2}\{i(s+1)+1\}$ be  the ideal of the NRT poset $[m\cdot (s+1)]$ formed by the $m-1$ minimal elements of the first $m-1$ blocks. Let $A$ be an $OCA(N;t+1,m,s+1,v)$. For a fixed $\beta \in \{0,1,\ldots,v-1\}$, consider the array $A_{\beta}$ obtained by choosing the rows of $A$ such that each entry in the last column is equal to $\beta$, and deleting the $m-1$ columns labeled by the minimal elements of the ideal $I$ as well as the last column (i.e.~the one labeled by the maximal element $m(s+1)$ of the last block).  Then select $\beta \in \{0,1,\ldots,v-1\}$ that occurs the least number of times in the last column of $A$, to ensure $A_{\beta}$ has at most  $\lfloor \frac{N}{v} \rfloor$ rows. We claim that the array $A_{\beta}$ is an $OCA(M;t,m,s,v)$ such that $M\leq \lfloor \frac{N}{v}\rfloor$.

\begin{figure}[t]
\centering
\includegraphics{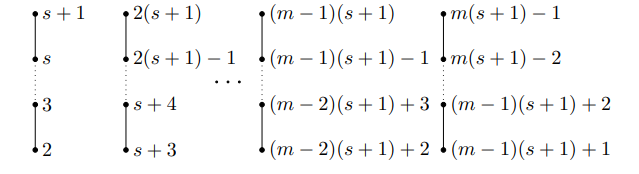}
\caption{$P\subset [m\cdot (s+1)]$ is an NRT poset $\mathcal{R}[m \cdot s]$.}
\label{fig1}
\end{figure} 

Let $P$ be the NRT poset $\mathcal{R}[m \cdot s]$ given in Fig.~\ref{fig1}. Let $J$ be an anti-ideal of $P$ of size $t$. The set $J'=J\cup \{m(s+1)\}$ is an anti-ideal of the NRT poset $[m\cdot (s+1)]$ of size $t+1$. The columns of $A$ labeled by $J'$ cover all the $(t+1)$-tuples over $\{0,1,\ldots,v-1\}$ at least once. Now, looking at the rows where the last entry is $\beta$ in $A$, the columns labeled by $J$ cover all the $t$-tuples over $\{0,1,\ldots,v-1\}$ at least once. Therefore, $A_{\beta}$ is an $OCA(M;t,m,s,v)$ such that $M\leq \lfloor \frac{N}{v}\rfloor$.
\end{proof}

Theorem~\ref{t1} below gives a construction for an $OCA(M;t,m+1,s,v)$ obtained by adding a block of $s$ columns to an $OCA(N;t,m,s,v)$ and additional rows, in the same spirit as some existing constructions of covering arrays which add a column.
However, in the case of ordered covering arrays this is a more complicated task. Before we prove Theorem~\ref{t1}, we need to construct an array and give a technical lemma showing the properties of this new array.

\begin{construction}\label{construction1}
Let $P$ be an NRT poset $\mathcal{R}[2\cdot s]$  with blocks $B_1=\{b_1, \ldots, b_s\}$ and $B_2=\{b_{s+1}, \ldots, b_{2s}\}$, where $b_1 \preceq b_2 \preceq\cdots \preceq b_s$ and $b_{s+1} \preceq b_{s+2}\cdots \preceq b_{2s}$.
Let $j\leq 2s$, and define an array $T^j$ with $2s$ columns labeled by $b_1,\ldots, b_{2s}$ in this order, and with rows specified as follows.
Each row of $T^j$ is indexed by each tuple $x=(x_1,\ldots,x_j)\in \{0,\ldots,v-1\}^j$ such that $(x_1,\ldots,x_{\lfloor \frac {j}{2} \rfloor}) \not= (x_j,x_{j-1},\ldots, x_{\lceil \frac {j}{2} \rceil+1})$, and
row $T^j_x=[a_1,a_2,\ldots,a_s,a_{s+1},a_{s+2},\ldots,a_{2s}]$ where:
\begin{enumerate}
\item $a_{s-i+1}=x_{i}$, if $1\leq i \leq min\{j,s\}$;
\item $a_{2s-i+1}=x_{j-i+1}$, if $1\leq i \leq min\{j,s\}$ and;
\item $a_1, \ldots, a_{s-min\{j,s\}}$ as well as $a_{s+1}, \ldots, a_{2s-min\{j,s\}}$ are set arbitrarily.
\end{enumerate}
\end{construction}

\begin{lemma} \label{tough}
The array $T^j$ given in Construction~\ref{construction1} has $(v^j-v^{\lceil \frac {j}{2} \rceil})$ rows.
Moreover, for every anti-ideal $J$ of $P$ with \textbar$J$\textbar $=j$, letting $j_1=$\textbar $B_1\cap J$\textbar and $j_2=$\textbar $B_2\cap J$\textbar and $\overline{t}=\min\{ j_1,j_2\}$, we have that the subarray of $T^j$ labeled by $J$ contains every tuple $(y,z)=(y_1, \ldots, y_{j_1}, z_1,  \ldots, z_{j_2})$ of $\{0,\ldots,v-1\}^j$ such that $(y_{j_1-\overline{t}+1}, y_{j_1-\overline{t}+2}, \ldots, y_{j_1})\not=(z_{j_2-\overline{t}+1}, z_{j_2-\overline{t}+2}, \ldots, z_{j_2})$ as a row.
\end{lemma}

\begin{proof}
A tuple  $(z_{1},\ldots,z_{n})$ is palindromic if $(z_{1},\ldots,z_{n})=(z_{n},\ldots,z_{1})$.
The array $T^j$ has one row per $x=(x_1,\ldots,x_j)\in \{0,\ldots,v-1\}^j$, such that row $T^j_x$ is not a palindromic tuple. There are $v^{\lceil \frac {j}{2} \rceil}$ palindromic tuples, so the total number of rows in $T^j$ is $(v^j-v^{\lceil \frac {j}{2} \rceil})$.
An anti-ideal $J$ is always formed by ``right-justified" subsets of block elements. Then  $J=\{b_{s-(j_1-1)}, \ldots, b_s\} \cup \{b_{2s-(j_2-1)}, \ldots, b_{2s}\}$, where $j_1+j_2=$\textbar$J$\textbar.
By construction, for row $T^j_x$, we assign different elements from the set $\{x_1,x_2,\ldots,x_j\}$ to these positions. Thus every $j$-tuple of the form $(y_1,\ldots,y_{j_1}, z_1, \ldots, z_{j_2})$ where $y=(y_1,\ldots,y_{j_1})$ is not a suffix of $z=(z_1, \ldots,z_{j_2})$ nor $z$ is a suffix of $y$ appear in these positions in some row of $T^j$.
\end{proof}

We illustrate the construction from Lemma~\ref{tough} in the following example.

\begin{example}
Take $v=3$, $j=3$, and the NRT poset $[2\cdot 5]$ in Lemma~\ref{tough}. We obtain array $T^3$ as below, where $\ast$ denotes positions of the array filled arbitrarily. The anti-ideals of cardinality $j=3$ for the set of columns $\{1,\ldots,10\}$ of $T^3$  are
$\{3,4,5\}$, $\{4,5,10\}$, $\{5,9,10\}$, and $\{8,9,10\}$.

Take $v=2$, $j=4$, and the NRT poset $[2\cdot 2]$ in Lemma~\ref{tough}.  We obtain array $T^{4}$ as below. There is only one  anti-ideal of cardinality $j=4$ corresponding to columns of $T^4$, namely  $\{1,2,3,4\}$.

\begin{minipage}{0.6\textwidth}
\[
T^3=\left[\begin{array}{ccccc|ccccc}
\ast & \ast & 0 & 0 & 1 & \ast & \ast & 1 & 0 & 0 \\
\ast & \ast & 0 & 0 & 2 & \ast & \ast & 2 & 0 & 0 \\
\ast & \ast & 0 & 1 & 1 & \ast & \ast & 1 & 1 & 0 \\
\ast & \ast & 0 & 1 & 2 & \ast & \ast & 2 & 1 & 0 \\
\ast & \ast & 0 & 2 & 1 & \ast & \ast & 1 & 2 & 0 \\
\ast & \ast & 0 & 2 & 2 & \ast & \ast & 2 & 2 & 0 \\
\ast & \ast & 1 & 0 & 0 & \ast & \ast & 0 & 0 & 1 \\
\ast & \ast & 1 & 0 & 2 & \ast & \ast & 2 & 0 & 1 \\
\ast & \ast & 1 & 1 & 0 & \ast & \ast & 0 & 1 & 1 \\
\ast & \ast & 1 & 1 & 2 & \ast & \ast & 2 & 1 & 1 \\
\ast & \ast & 1 & 2 & 0 & \ast & \ast & 0 & 2 & 1 \\
\ast & \ast & 1 & 2 & 2 & \ast & \ast & 2 & 2 & 1 \\
\ast & \ast & 2 & 0 & 0 & \ast & \ast & 0 & 0 & 2 \\
\ast & \ast & 2 & 0 & 1 & \ast & \ast & 1 & 0 & 2 \\
\ast & \ast & 2 & 1 & 0 & \ast & \ast & 0 & 1 & 2 \\
\ast & \ast & 2 & 1 & 1 & \ast & \ast & 1 & 1 & 2 \\
\ast & \ast & 2 & 2 & 0 & \ast & \ast & 0 & 2 & 2 \\
\ast & \ast & 2 & 2 & 1 & \ast & \ast & 1 & 2 & 2\\
\end{array}
\right]
\]
\end{minipage}
\begin{minipage}{0.3\textwidth}
\[
T^{4}=\left[
\begin{array}{cc|cc}
0 & 0 & 0 & 1\\
0 & 0 & 1 & 0\\
0 & 0 & 1 & 1\\
0 & 1 & 0 & 0\\
0 & 1 & 1 & 0\\
0 & 1 & 1 & 1\\
1 & 0 & 0 & 0\\
1 & 0 & 0 & 1\\
1 & 0 & 1 & 1\\
1 & 1 & 0 & 0\\
1 & 1 & 0 & 1\\
1 & 1 & 1 & 0\\
\end{array}
\right]
\]
\end{minipage}

For each of these arrays we can verify  that the subarray labeled by each anti-ideal visits every tuple in the alphabet, except for those tuples where one of the two parts is a suffix of the other part.
For example, columns $\{4,5,10\}$ does not contain tuple $(01$\textbar $1)$ as a row, since 1 is a suffix of 01, but contains tuple $(01$\textbar$0)$, since 0 is not a suffix of 01 and 01 is not a suffix of 0.
\end{example}

We are now ready to prove Theorem~\ref{t1}, that builds an $OCA(M;t,m+1,s,v)$ obtained from an $OCA(N;t,m,s,v)$ by adding a block of $s$ columns  and  additional rows using cartesian product.

\begin{theorem}\label{t1}
For $s\leq t$ and $k=\min \{2s,t\}$, $OCAN(t,m+1,s,v)$ is less than or equal to
\[
OCAN(t,m,s,v)+\sum_{j=2}^{k} OCAN(t-j,m-1,s,v)\cdot (v^{j}-v^{\lceil\frac{j}{2}\rceil}).
\]
\end{theorem}

\begin{proof}
Let $B_{0}, B_{1}, \ldots,B_{m-1}, B_{m}$ be the blocks of the NRT poset $[(m+1)\cdot s]$.
Let $A$ be an $OCA(N;t,m,s,v)$ over $\{0,1,\ldots, v-1\}$. Let us consider the array $A'$ formed by concatenating $A$ with the subarray of $A$  formed by the last $s$ columns of $A$.  We add rows of other arrays below $A'$ as described next.
Let $k=\min \{2s,t\}$. For each  $j \in \{2,\ldots, k\}$, we build $C^j$ which is the cartesian product of rows of an \linebreak $OCA(N_{j};t-j,m-1,s,v)$ with rows of the array $T^j$ constructed in Lemma~\ref{tough}. We note that if $j=t$, then the entries of an $OCA(1;0,m-1,s,v)$ can be set arbitrarily.
Each array $C^{j}$ has $N_{j}\cdot (v^{j}-v^{\lceil\frac{j}{2}\rceil})$ rows.

Let $\mathcal{A}$ be the array obtained by vertically juxtaposing $A'$, $C^{2}$, $\ldots$, $C^{k}$, and $M=N+\sum_{j=2}^{k}N_{j}\cdot (v^{j}-v^{\lceil\frac{j}{2}\rceil})$. We claim that $\mathcal{A}$ is an  $OCA(M;t,m+1,s,v)$.
Let $J$ be an anti-ideal of size $t$ of the NRT poset $[(m+1)\cdot s]$. We divide the proof into three cases.

\begin{enumerate}
\item If $J\cap B_{m-1}=\emptyset$ and $J\cap B_{m}=\emptyset$, then the $t$-tuples over $\{0,1,\ldots,v-1\}$ are covered by  the columns of $A$ labeled by $J$, which are present in $A'$.
\item If $J\cap B_{m-1}=\emptyset$ or $J\cap B_{m}=\emptyset$, then the $t$-tuples over $\{0,1,\ldots,v-1\}$ are covered by  the columns of $A'$ labeled by $J$ since the last $s$ columns of $A'$ contain all tuples in $A$ corresponding to block $B_{m-1}$ which together with other blocks already satisfied coverage in $A$.
\item If $J\cap B_{m-1}\neq \emptyset$ and $J\cap B_{m}\neq\emptyset$, then \textbar$J\cap (B_{m-1}\cup B_{m})$\textbar$=j$ for $j\in\{2,\ldots,k\}$, where $k=\min \{2s,t\}$. For each $j\in\{2,\ldots,k\}$, we show that the juxtaposition of arrays $A'$ and $C^{j}$ constructed above cover the $t$-tuples over $\{0,1,\ldots,v-1\}$ in the subarray corresponding to the columns $J$.
We regard each $t$-tuple over $\{0,1,\ldots,v-1\}$ as $(x,y_1,y_2)$, where $x$ is a $(t-j)$-tuple and $(y_1,y_2)$ is a $j$-tuple, with component of $x$ corresponding to $J\setminus(B_{m-1}\cup B_{m})$,  the components of $y_1$ corresponding to  $J\cap B_{m-1}$ and the components of $y_2$ corresponding to $J\cap B_{m}$.
If $y_1$ is a suffix of $y_2$, since $(x,y_2)$ is a tuple covered in $A$  for columns corresponding to $J\setminus (B_{m-1}\cup B_m) $ and $B_{m-1}$, by construction of $A'$ we have that $(x,y_1,y_2)$ is covered in $A'$ for columns corresponding to $J\setminus (B_{m-1}\cup B_m) $, $B_{m-1}$ and $B_m$. Similarly, if $y_2$ is a suffix of $y_1$, we get coverage in $A'$.
On the contrary, if $\overline{y_1}$ and $\overline{y_2}$ are suffixes of largest possible common size for $y_1$ and $y_2$ with the property that $\overline{y_1}\not=\overline{y_2}$, then $(y_1,y_2)$ must be covered in $T^j$. Since $x$ is covered in columns corresponding to $J\setminus (B_{m-1}\cup B_m) $ of the $OCA(N_{j};t-j,m-1,s,v)$, then $(x,y_1,y_2)$ is covered in columns  of $C^j$, which is the cartesian product of an $OCA(N_{j};t-j,m-1,s,v)$ with  $T^j$.
\end{enumerate}
Therefore, $\mathcal{A}$ is an $OCA(M;t,m+1,s,v)$, where $M=\displaystyle N+\sum_{j=2}^{k}N_{j}\cdot (v^{j}-v^{\lceil\frac{j}{2}\rceil})$.
\end{proof}

In the next example, we illustrate the construction given by Theorem \ref{t1}.
\begin{example}
 For $t=s=3$, $m=3$, and $v=2$,
\[
OCAN(3,4,3,2)\leq OCAN(3,3,3,2)+ OCAN(1,2,3,2)\cdot 2 + OCAN(0,2,3,2)\cdot 4.
\]
The $OCA(16;3,4,3,2)$, $\mathcal{A}$, constructed in Theorem \ref{t1} is shown below.
We explain each part of the array $\mathcal{A}$. Let $B_{0}=\{1,2,3\}, B_{1}=\{4,5,6\}, B_{2}=\{7,8,9\}, B_{3}=\{10,11,12\}$ be the blocks of the NRT poset $[4\cdot 3]$.
By Eq.~(\ref{eq1}),  $OCAN(3,3,3,2)=8$ and its corresponding array is given by Part I in the columns labeled by the blocks $B_{0}$, $B_{1}$ and $B_{2}$. In Part I, we have the array $A'$, and the columns labeled by the blocks $B_{2}$ and $B_{3}$ are repeated in $A'$. For $j\in\{2,3\}$, we build $T^{j}$ as in Lemma~\ref{tough}.
In Part II and III,  we have the array $C_{2}$ which is the cartesian product of an $OCA(2;1,2,3,2)=\left(\begin{array}{c|c}000&000\\111&111\\\end{array}\right)$ with
$T^{2}=\left(\begin{array}{c|c}*10& *01\\ *01&*10\\\end{array}\right)$.
In Part IV, we have the array $C_{3}$. The symbol $\ast$ in columns labeled by $B_{0}$ and $B_{1}$ can be set arbitrarily since an $OCA(N;0,2,3,2)$ does not need to have any special property. Finally, the columns labeled by $B_{2}$ and $B_{3}$ in Part IV give the array $T^{3}$.
\begin{center}
\begin{minipage}{0.1\textwidth}
\[
\begin{array}{r}
   \\
   \\
   \\
  \textrm{Part I} \\
   \\
   \\
   \\
   \\
   \textrm{Part II} \\
   \\
  \textrm{Part III} \\
   \\
   \\
  \textrm{Part IV}  \\
   \\
\end{array}
\]
\end{minipage}
\begin{minipage}{0.7\textwidth}
\centering
\[
\begin{array}{c}
\hspace{4mm}\begin{array}{cccc}
 B_{0} \ \ \  & \ \ \  B_{1} \ \ \  & \ \ \  B_{2} \ \ \  &  \ \ \ B_{3}   \ \ \ \   \\
\end{array}\\
\left[\begin{array}{ccc|ccc|ccc|ccc}
 1 & 0 & 0 & 1 & 1 & 1 & 0 & 1 & 1 & 0 & 1 & 1 \\
 0 & 0 & 1 & 0 & 1 & 1 & 1 & 0 & 0 & 1 & 0 & 0 \\
 0 & 1 & 1 & 1 & 0 & 1 & 1 & 0 & 1 & 1 & 0 & 1 \\
 1 & 1 & 1 & 0 & 1 & 0 & 1 & 1 & 0 & 1 & 1 & 0 \\
 1 & 1 & 0 & 0 & 0 & 1 & 0 & 1 & 0 & 0 & 1 & 0 \\
 1 & 0 & 1 & 1 & 0 & 0 & 1 & 1 & 1 & 1 & 1 & 1 \\
 0 & 1 & 0 & 1 & 1 & 0 & 0 & 0 & 1 & 0 & 0 & 1 \\
 0 & 0 & 0 & 0 & 0 & 0 & 0 & 0 & 0 & 0 & 0 & 0 \\ \hline
 0 & 0 & 0 & 0 & 0 & 0 & \ast & 1 & 0 & \ast & 0 & 1 \\
 0 & 0 & 0 & 0 & 0 & 0 & \ast & 0 & 1 & \ast & 1 & 0 \\ \hline
 1 & 1 & 1 & 1 & 1 & 1 & \ast & 1 & 0 & \ast & 0 & 1 \\
 1 & 1 & 1 & 1 & 1 & 1 & \ast & 0 & 1 & \ast & 1 & 0 \\ \hline
 \ast & \ast & \ast & \ast & \ast & \ast & 1 & 0 & 0 & 0 & 0 & 1 \\
 \ast & \ast & \ast & \ast & \ast & \ast & 1 & 1 & 0 & 0 & 1 & 1 \\
 \ast & \ast & \ast & \ast & \ast & \ast & 0 & 0 & 1 & 1 & 0 & 0 \\
 \ast & \ast & \ast & \ast & \ast & \ast & 0 & 1 & 1 & 1 & 1 & 0
\end{array}\right].
\end{array}
\]
\end{minipage}
\end{center}
\end{example}

As discussed before, for $v$ a prime power and $t\geq 3$, Castoldi et al. \cite[Theorem 3]{castoldi2017ordered} shows a construction for an $OOA(v^t;t,v+1,t,v)$, which yields Eq.~(\ref{eq1}). Using this result and Corollary \ref{thm6}  to create ingredients for Theorem~\ref{t1}, we establish an upper bound for $OCAN(t,v+2,t,v)$.

\begin{corollary}\label{corprimepower}
Let $v$ be a prime power.
\begin{enumerate}
\item[(1)] If $t$ is odd, then $OCAN(t,v+2,t,v)\leq v^t (t- \frac{2}{1-v} (\frac{1}{v^{(t-1)/2}} -1) )$.
\item[(2)] If $t$ is even, then $OCAN(t,v+2,t,v)\leq v^t (t-\frac{2}{1-v} (\frac{1}{v^{(t-2)/2}} -1)  - \frac{1}{v^{t/2}} )$.
\end{enumerate}
\end{corollary}

\begin{proof}
For $v$ a prime power and $t\geq 3$,   $OCAN(t,v+1,t,v)=v^{t}$ by Eq.~(\ref{eq1}). For $j \in \{2, \ldots,t-3\}$ and since $t-j\leq t$, Proposition \ref{prop8} item $(2)$ yields  $OCAN(t-j,v,t,v)=OCAN(t-j,v,t-j,v)$. By Eq.~(\ref{eq1}), we obtain $OCAN(t-j,v,t-j,v)=v^{t-j}$. For $j=t-2$, Corollary \ref{thm6} implies that $OCAN(2,v,t,v)=OCAN(2,v,2,v)=CAN(2,v,v)$. Bush's construction \cite[Theorem 3.1]{hedayat2012orthogonal} gives $CAN(2,v,v)=v^2$. For $j=t-1$, $OCAN(1,v,t,v)=v$, and for $j=t$, $OCAN(0,v,t,v)=1$. In summary, $OCAN(t-j,v,t,v)=v^{t-j}$, for all $j \in \{2, \ldots,t\}$.
Using Theorem \ref{t1}, we obtain
\begin{eqnarray*}
OCAN(t,v+2,t,v)& \leq & v^t + \sum_{j=2}^t v^{t-j} (v^j-v^{\lceil\frac{j}{2}\rceil} ) \\
& = & t v^t -\sum_{j=2}^t v^{t-\lfloor\frac{j}{2}\rfloor}\\
& = & v^t \left(t-\sum_{j=2}^t \frac{1}{v^{\lfloor\frac{j}{2}\rfloor}}\right).
\end{eqnarray*}

We now consider two cases.
\begin{enumerate}
\item If $t$ is odd, then 
\[
\sum_{j=2}^t \frac{1}{v^{\lfloor\frac{j}{2}\rfloor}}=2\cdot \sum_{j=1}^{(t-1)/2} \frac{1}{v^{j}}=\frac{2}{1-v}\left(\frac{1}{v^{(t-1)/2}} -1\right).
\] 
Therefore $OCAN(t,v+2,t,v)\leq v^t (t- \frac{2}{1-v} (\frac{1}{v^{(t-1)/2}} -1) )$.
\item If $t$ is even, then 
\[
\sum_{j=2}^t \frac{1}{v^{\lfloor\frac{j}{2}\rfloor}}=\sum_{j=2}^{t-1} \frac{1}{v^{\lfloor\frac{j}{2}\rfloor}}+\frac{1}{v^{t/2}}.
\] 
Since $t-1$ is odd, 
\[
\sum_{j=2}^{t-1} \frac{1}{v^{\lfloor\frac{j}{2}\rfloor}}=\frac{2}{1-v}\left(\frac{1}{v^{(t-2)/2}} -1\right).
\] 
Therefore $OCAN(t,v+2,t,v)\leq$  $v^t (t-\frac{2}{1-v} (\frac{1}{v^{(t-2)/2}} -1)  - \frac{1}{v^{t/2}} )$.
\end{enumerate}
\end{proof}

\section{New upper bounds on covering codes in NRT spaces}\label{sec5}

In this section, we improve the general upper bound on $K_{q}^{\mathcal{R}}(m,s,R)$ given in Proposition \ref{triviais} for suitable
values of $m$ and $R$ by constructing new covering codes. 
For this purpose, we start from the covering code used to prove Proposition \ref{triviais} and our strategy consists of 
modifying some of its codewords to reduce the size of the covering code. Table \ref{tab3}, at the end of this section, compares
the upper bounds obtained in this section with those in the tables in \cite{castoldi2015covering}.

In order to facilitate the readability of the arguments, we represent a vector $x=$  $(x_{1},\ldots,x_{ms})$ $\in \mathbb{Z}_{q}^{ms}$ as a matrix:
\[ x=\left[
\begin{array}{cccc}
  x_{s} & x_{2s} & \cdots & x_{ms} \\
  \vdots & \vdots &  & \vdots \\
  x_{2} & x_{s+2} & \cdots & x_{(m-1)s+2} \\
  x_{1} & x_{s+1} & \cdots & x_{(m-1)s+1}
\end{array}
\right].
\]

\subsection{New covering codes for $m$ odd}\label{sec51}

Let $k$ be a positive integer. We first focus on the case where $m\geq 3$ is odd, $m=2k+1$, and the radius is $(k+1)s-j$, for $j=1,\ldots,s$.

\begin{theorem}\label{bigcaseodd}
Let $q$, $s$ and $k$ be positive integers such that $q\geq 2$, $s\geq 3$ and $k\geq 1$. For $j=1,\ldots,s$,  $K_{q}^{\mathcal{R}}(2k+1,s,(k+1)s-j)\leq q^{ks+j}-q^{k(s-2)+j}(q^k-1)$.
\end{theorem}

\begin{proof}
The general upper bound for $K_{q}^{\mathcal{R}}(2k+1,s,(k+1)s-j)$ is $q^{ks+j}$, and
a $((k+1)s-j)$-covering $C$ of the NRT space $\mathbb{Z}_{q}^{(2k+1)s}$ of size $q^{ks+j}$ is formed by the codewords:
\[
c=\left[
\begin{array}{ccccccc}
  0 &  \cdots & 0 & \star & c_{(k+2)s} & \cdots & c_{(2k+1)s}\\
  0 &  \cdots & 0 & \star          & c_{(k+2)s-1} & \cdots & c_{(2k+1)s-1}\\
  \vdots &    & \vdots & \vdots & \vdots &  & \vdots \\
  0 &  \cdots & 0 & \star          & c_{(k+1)s+2} & \cdots & c_{2ks+2}\\
  0 &  \cdots & 0 & \star          & c_{(k+1)s+1} & \cdots & c_{2ks+1}
\end{array}
\right],
\]
where the column $\star$ uses the $s$-tuple $(0,\ldots,0,c_{(k+1)s-(j-1)},\ldots,c_{(k+1)s})$, filled from bottom to top, according to the proof of Proposition \ref{triviais}. We divide the proof into three steps. 

\emph{Step 1:} We partition the set $C$ into $q^{k(s-2)+j}$ parts indexed by the set $\mathbb{Z}_{q}^{k(s-2)+j}$.  
For each $z=({\bf z_{0}},z_{1},\ldots,z_{k(s-2)})\in \mathbb{Z}_{q}^{k(s-2)+j}$, where ${\bf z_{0}}$ is a $j$-tuple over $\mathbb{Z}_{q}$, let $C_{z}$ be the subset of $C$ formed by the codewords:
\[
c=\left[
\begin{array}{ccccccc}
  0 &  \cdots & 0 & \ast & z_{s-2} & \cdots & z_{k(s-2)}\\
  0 &  \cdots & 0 & \ast          & z_{s-3} & \cdots & z_{k(s-2)-1}\\
  \vdots &    & \vdots & \vdots & \vdots &  & \vdots \\
  0 &  \cdots & 0 & \ast & z_{1} & \cdots & z_{(k-1)(s-2)+1}\\
  0 &  \cdots & 0 & \ast & c_{(k+1)s+2} & \cdots & c_{2ks+2}\\
  0 &  \cdots & 0 & \ast & c_{(k+1)s+1} & \cdots & c_{2ks+1}
\end{array}
\right],
\]
where the column $\ast$ uses the $s$-tuple $(0,\ldots,0,{\bf z_{0}})$, filled from bottom to top.
For each $z\in \mathbb{Z}_{q}^{k(s-2)+j}$, write
\[
\mathcal{Z}_{z}=\mathbb{Z}_{q}^{ks}\times (\mathbb{Z}_{q}^{s-j}\times {\bf z_{0}}) \times \prod_{i=0}^{k-1} (\mathbb{Z}_{q}^{2}\times\{(z_{i(s-2)+1},\ldots,z_{i(s-2)+s-2})\}).
\]

The following properties hold:
\begin{itemize}
\item[(a)] $C_{z}\cap C_{z'}=\emptyset$ if and only if $z\neq z'$;
\item[(b)] \textbar$C_{z}$\textbar$=q^{2k}$ for all $z\in \mathbb{Z}_{q}^{k(s-2)+j}$;
\item[(c)] $\displaystyle C=\cup_{z\in \mathbb{Z}_{q}^{k(s-2)+j}} C_{z}$;
\item[(d)] $C_{z}$ is a $((k+1)s-j)$-covering of the  NRT space $\mathcal{Z}_{z}$  over the NRT poset $[(2k+1)\cdot s]$.
\end{itemize}
We observe that $(a)$, $(b)$ and $(c)$ tell us the set $\{C_{z}:z\in \mathbb{Z}_{q}^{k(s-2)+j}\}$ is a partition of $C$.

\emph{Step 2:} For each $z\in \mathbb{Z}_{q}^{k(s-2)+j}$, we construct a new set $A_{z}$ from $C_{z}$
such that $A_{z}$ is a $((k+1)s-j)$-covering of the NRT space $\mathcal{Z}_{z}$.
For each $c\in C_{z}$ such that   $(c_{(k+1)s+1},c_{(k+2)s+1},\ldots,c_{2ks+1})$ $=$ $(c_{(k+1)s+2},c_{(k+2)s+2},\ldots,c_{2ks+2})$ $= (0,\ldots,0)$ or $(c_{(k+1)s+2},c_{(k+2)s+2},\ldots,c_{2ks+2})\neq $ $(0,0,\ldots,0)$, define
\[
c'=\left[
\begin{array}{ccccccc}
  c_{(k+1)s+2} & \cdots & c_{2ks+2} & \ast & z_{s-2} & \cdots & z_{k(s-2)}\\
  c_{(k+1)s+1} & \cdots & c_{2ks+1} & \ast          & z_{s-3} & \cdots & z_{k(s-2)-1}\\
  0 & \cdots & 0 & \ast          & z_{s-4} & \cdots & z_{k(s-2)-2}\\
  \vdots &    & \vdots & \vdots & \vdots &  & \vdots \\
  0 &  \cdots & 0 & \ast & z_{1} & \cdots & z_{(k-1)(s-2)+1}\\
  0 &  \cdots & 0 & \ast & c_{(k+1)s+2} & \cdots & c_{2ks+2}\\
  0 &  \cdots & 0 & \ast & c_{(k+1)s+1} & \cdots & c_{2ks+1}
\end{array}
\right],
\]
where the column $\ast$ uses the $s$-tuple $(0,\ldots,0,{\bf z_{0}})$, filled from bottom to top.

Let $A_{z}$ be the set of codewords $c'$ defined above. The types of codewords in $C_{z}$ that we are not using to define $A_{z}$ are those such that $(c_{(k+1)s+1},c_{(k+2)s+1},\ldots,$ $c_{2ks+1})$ $\neq (0,0,\ldots,0)$ and $(c_{(k+1)s+2},c_{(k+2)s+2},\ldots,c_{2ks+2})= (0,0,\ldots,0)$. There are $q^k-1$ such codewords and $A_{z}$ has size $q^{2k}-(q^k-1)$.

\emph{Step 3:} It remains to show that the set $A_{z}$ is a $((k+1)s-j)$-covering of the NRT space $\mathcal{Z}_{z}$. The proof is divided into three cases.

Indeed, for $x \in \mathcal{Z}_{z}$, we know that $x$ and $c'\in A_{z}$ coincide in those $k(s-2)+j$ positions that are equal to $z$. We highlight in bold the $2k$ positions in each codeword that coincide with the respective positions in $x$.
\begin{enumerate}
\item If $(x_{(k+1)s+2},x_{(k+2)s+2},\ldots,x_{2ks+2})\neq (0,0,\ldots,0)$, then $x$ is covered by the following element of $A_{z}$:
\[
\left[
\begin{array}{ccccccc}
  x_{(k+1)s+2} & \cdots & x_{2ks+2} & \ast & z_{s-2} & \cdots & z_{k(s-2)}\\
  x_{(k+1)s+1} & \cdots & x_{2ks+1} & \ast          & z_{s-3} & \cdots & z_{k(s-2)-1}\\
  0 & \cdots & 0 & \ast          & z_{s-4} & \cdots & z_{k(s-2)-2}\\
  \vdots &    & \vdots & \vdots & \vdots &  & \vdots \\
  0 &  \cdots & 0 & \ast & z_{1} & \cdots & z_{(k-1)(s-2)+1}\\
  0 &  \cdots & 0 & \ast & \boldsymbol{ x_{(k+1)s+2}} & \cdots & \boldsymbol{ x_{2ks+2}}\\
  0 &  \cdots & 0 & \ast & \boldsymbol{ x_{(k+1)s+1}} & \cdots & \boldsymbol{ x_{2ks+1}}
\end{array}
\right].
\]
\item If $(x_{s},x_{2s}\ldots,x_{ks})\neq$ $ (0,0,\ldots,0)$, then $x$ is covered by the following element of $A_{z}$:
\[
\left[
\begin{array}{ccccccc}
  \boldsymbol{x_{s}} & \cdots & \boldsymbol{ x_{ks}} & \ast & z_{s-2} & \cdots & z_{k(s-2)}\\
  \boldsymbol{ x_{s-1}} & \cdots & \boldsymbol{ x_{ks-1}} & \ast          & z_{s-3} & \cdots & z_{k(s-2)-1}\\
  0 & \cdots & 0 & \ast          & z_{s-4} & \cdots & z_{k(s-2)-2}\\
  \vdots &    & \vdots & \vdots & \vdots &  & \vdots \\
  0 &  \cdots & 0 & \ast & z_{1} & \cdots & z_{(k-1)(s-2)+1}\\
  0 &  \cdots & 0 & \ast & x_{s} & \cdots & x_{ks}\\
  0 &  \cdots & 0 & \ast & x_{s-1} & \cdots & x_{ks-1}
\end{array}
\right].
\]
\item If $(x_{(k+1)s+2},x_{(k+2)s+2},\ldots,x_{2ks+2})=$ $ (x_{s},x_{2s}\ldots,x_{ks})$ $= (0,0,\ldots,0)$, then $x$ is covered by the following element of $A_{z}$:
\[
\left[
\begin{array}{ccccccc}
  {\bf 0} & \cdots & {\bf 0} & \ast & z_{s-2} & \cdots & z_{k(s-2)}\\
  0 & \cdots & 0 & \ast          & z_{s-3} & \cdots & z_{k(s-2)-1}\\
  0 & \cdots & 0 & \ast          & z_{s-4} & \cdots & z_{k(s-2)-2}\\
  \vdots &    & \vdots & \vdots & \vdots &  & \vdots \\
  0 &  \cdots & 0 & \ast & z_{1} & \cdots & z_{(k-1)(s-2)+1}\\
  0 &  \cdots & 0 & \ast & {\bf 0} & \cdots & {\bf 0}\\
  0 &  \cdots & 0 & \ast & 0 & \cdots & 0
\end{array}
\right].
\]
\end{enumerate}

Therefore, the set $A=\bigcup_{z\in \mathbb{Z}_{q}^{k(s-2)+j}}A_{z}$ of size $q^{ks+1}-q^{k(s-2)+j}(q^{k}-1)$ is a $((k+1)s-j)$-covering of the NRT
space $\mathbb{Z}_{q}^{(2k+1)s}$. 
\end{proof}

In addition to Theorem \ref{bigcaseodd} being a refinement of the bound given in Proposition \ref{triviais},
for $k=j=1$, Theorem \ref{bigcaseodd} improves the bound $q^{s+1}-q$ given in \cite[Theorem 4]{castoldi2019cai} item $(2)$.

\begin{corollary}\label{caseodd}
For  $q\geq 2$ and $s\geq 3$, $K_{q}^{\mathcal{R}}(3,s,2s-1)\leq q^{s+1}-q^{s-1}(q-1)$.
\end{corollary}

For $j=s$, we have the following particular case of Theorem \ref{bigcaseodd} which also improves upper bounds given in \cite{castoldi2015covering}, see Table \ref{tab3}.

\begin{corollary}\label{cases}
For $q\geq 2$ and $s\geq 3$,  $K_{q}^{\mathcal{R}}(2k+1,s,ks)\leq q^{(k+1)s}-q^{k(s-2)+s}(q^k-1)$.
\end{corollary}

\subsection{New covering codes for $m$ even} \label{sec52}

We show next Theorem \ref{caseven} which is a parallel result to Theorem \ref{bigcaseodd} for $m$ even, and generalizes \cite[Theorem 4]{castoldi2019cai} item $(1)$. The proof of Theorem \ref{caseven} follows the same method applied to prove Theorem \ref{bigcaseodd}.

\begin{theorem}\label{caseven}
For $q\geq 2$, $s\geq 3$ and $k\geq 1$, $K_{q}^{\mathcal{R}}(2k,s,ks)\leq q^{ks}-q^{k(s-2)}(q^{k}-1)$.
\end{theorem}

\begin{proof}
The general upper bound for $K_{q}^{\mathcal{R}}(2k,s,ks)$ is $q^{ks}$, and
a $ks$-covering $C$ of the NRT space $\mathbb{Z}_{q}^{2ks}$ of size $q^{ks}$ is formed by the codewords:
\[
c=\left[
\begin{array}{cccccc}
  0 &  \cdots & 0 & c_{(k+1)s} & \cdots & c_{2ks}\\
  \vdots &    & \vdots & \vdots &  & \vdots \\
  0 &  \cdots & 0 & c_{ks+2} & \cdots & c_{(2k-1)s+2}\\
  0 &  \cdots & 0 & c_{ks+1} & \cdots & c_{(2k-1)s+1}
\end{array}
\right],
\]
according to the proof of Proposition \ref{triviais}.
We divide the proof into three steps.

\emph{Step 1:} We partition the set $C$ into $q^{k(s-2)}$ parts indexed by the set $\mathbb{Z}_{q}^{k(s-2)}$. For each $z=(z_{1},\ldots,z_{k(s-2)})\in \mathbb{Z}_{q}^{k(s-2)}$, let $C_{z}$ be the subset of $C$ formed by the codewords:
\[
c=\left[
\begin{array}{cccccc}
  0 &  \cdots & 0 & z_{s-2} & \cdots & z_{k(s-2)}\\
  \vdots &    & \vdots & \vdots &  & \vdots \\
  0 &  \cdots & 0 & z_{1} & \cdots & z_{(k-1)(s-2)+1}\\
  0 &  \cdots & 0 & c_{ks+2} & \cdots & c_{(2k-1)s+2}\\
  0 &  \cdots & 0 & c_{ks+1} & \cdots & c_{(2k-1)s+1}
\end{array}
\right].
\]
For each $z\in \mathbb{Z}_{q}^{k(s-2)}$, define
\[
\mathcal{Z}_{z}=\mathbb{Z}_{q}^{ks}\times (\mathbb{Z}_{q}^{2}\times\{(z_{1},\ldots,z_{s-2})\})\times \ldots \times (\mathbb{Z}_{q}^{2}\times\{(z_{(k-1)(s-2)+1},\ldots,z_{k(s-2)})\}).
\]

The following properties hold:
\begin{itemize}
\item[(a)] $C_{z}\cap C_{z'}=\emptyset$ if and only if $z\neq z'$;
\item[(b)] \textbar $C_{z}$\textbar$=q^{2k}$ for all $z\in \mathbb{Z}_{q}^{k(s-2)}$;
\item[(c)] $\displaystyle C=\cup_{z\in \mathbb{Z}_{q}^{k(s-2)}} C_{z}$;
\item[(d)] $C_{z}$ is a $ks$-covering of the  NRT space  $\mathcal{Z}_{z}$ over the NRT poset $[2k\cdot s]$.
\end{itemize}
We note that $(a)$, $(b)$ and $(c)$ tell us the set $\{C_{z}:z\in \mathbb{Z}_{q}^{k(s-2)}\}$ is a partition of $C$.

\emph{Step 2:} For each $z\in \mathbb{Z}_{q}^{k(s-2)}$, we construct a new set $A_{z}$ from $C_{z}$
such that $A_{z}$ is a $ks$-covering of the NRT space $\mathcal{Z}_{z}$.
For each $c\in C_{z}$ such that  $(c_{ks+2},c_{(k+1)s+2},\ldots,$ $c_{(2k-1)s+2})$ $\neq (0,0,\ldots,0)$ or $(c_{ks+1},c_{(k+1)s+1},\ldots,c_{(2k-1)s+1})$ $=$ $(c_{ks+2},c_{(k+1)s+2},$ $\ldots,c_{(2k-1)s+2})= $ $(0,0,\ldots,0)$, define
\[
c'=\left[
\begin{array}{cccccc}
  c_{ks+2} & \cdots & c_{(2k-1)s+2} & z_{s-2} & \cdots & z_{k(s-2)}\\
  c_{ks+1} & \cdots & c_{(2k-1)s+1} & z_{s-3} & \cdots & z_{k(s-2)-1}\\
  0 & \cdots & 0 & z_{s-4} & \cdots & z_{k(s-2)-2}\\
  \vdots &    & \vdots & \vdots &  & \vdots \\
  0 &  \cdots & 0 & z_{1} & \cdots & z_{(k-1)(s-2)+1}\\
  0 &  \cdots & 0 & c_{ks+2} & \cdots & c_{(2k-1)s+2}\\
  0 &  \cdots & 0 & c_{ks+1} & \cdots & c_{(2k-1)s+1}
\end{array}
\right].
\]
Let $A_{z}$ be the set of codewords $c'$ defined above. The types of codewords in $C_{z}$ that we are not using to define $A_{z}$ are those such that
$(c_{ks+1},c_{(k+1)s+1},\ldots,$ $c_{(2k-1)s+1})\neq (0,0,\ldots,0)$ and $(c_{ks+2},c_{(k+1)s+2},\ldots,$ $c_{(2k-1)s+2})$ $= (0,0,\ldots,0)$, and there are \linebreak $q^k-1$ such codewords.
Hence, $A_{z}$ has size $q^{2k}-(q^k-1)$.

\emph{Step 3:} We now show that the set $A_{z}$ is a $ks$-covering of the NRT space $\mathcal{Z}_{z}$. We divide the proof into three cases.
Indeed, for $x \in \mathcal{Z}_{z}$, we know that $x$ and $c'\in A_{z}$ coincide in those $k(s-2)$ positions that are equal to $z$. We highlight in bold the $2k$ positions in each codeword that coincide with the respective positions in $x$.
\begin{enumerate}
\item If $(x_{ks+2},x_{(k+1)s+2},\ldots,x_{(2k-1)s+2})\neq (0,0,\ldots,0)$, then $x$ is covered by the following element of $A_{z}$:
\[
\left[
\begin{array}{cccccc}
  x_{ks+2} & \cdots & x_{(2k-1)s+2} & z_{s-2} & \cdots & z_{k(s-2)}\\
  x_{ks+1} & \cdots & x_{(2k-1)s+1} & z_{s-3} & \cdots & z_{k(s-2)-1}\\
  0 & \cdots & 0 & z_{s-4} & \cdots & z_{k(s-2)-2}\\
  \vdots &    & \vdots & \vdots &  & \vdots \\
  0 &  \cdots & 0 & z_{1} & \cdots & z_{(k-1)(s-2)+1}\\
  0 &  \cdots & 0 & {\bf x_{ks+2}} & \cdots & {\bf x_{(2k-1)s+2}}\\
  0 &  \cdots & 0 & {\bf x_{ks+1}} & \cdots & {\bf x_{(2k-1)s+1}}
\end{array}
\right].
\]
\item If  $(x_{s},x_{2s}\ldots,x_{ks})\neq (0,0,\ldots,0)$, then $x$ is covered by the following element of $A_{z}$:
\[
\left[
\begin{array}{cccccc}
  {\bf x_{s}} & \cdots & {\bf x_{ks}} & z_{s-2} & \cdots & z_{k(s-2)}\\
  {\bf x_{s-1}} & \cdots & {\bf x_{ks-1}} & z_{s-3} & \cdots & z_{k(s-2)-1}\\
   0 & \cdots & 0 & z_{s-4} & \cdots & z_{k(s-2)-2}\\
  \vdots &    & \vdots & \vdots &  & \vdots \\
  0 &  \cdots & 0 & z_{1} & \cdots & z_{(k-1)(s-2)+1}\\
  0 &  \cdots & 0 & x_{s} & \cdots & x_{ks}\\
  0 &  \cdots & 0 & x_{s-1} & \cdots & x_{ks-1}
\end{array}
\right].
\]
\item If $(x_{ks+2},x_{(k+1)s+2},\ldots,x_{(2k-1)s+2})= (x_{s},x_{2s}\ldots,x_{ks})= (0,0,\ldots,0)$, then $x$ is covered by the following element of $A_{z}$:
\[
\left[
\begin{array}{cccccc}
  {\bf 0} & \cdots & {\bf 0} & z_{s-2} & \cdots & z_{k(s-2)}\\
  0 & \cdots & 0 & z_{s-3} & \cdots & z_{k(s-2)-1}\\
  \vdots &    & \vdots & \vdots &  & \vdots \\
  0 &  \cdots & 0 & z_{1} & \cdots & z_{(k-1)(s-2)+1}\\
  0 &  \cdots & 0 & {\bf 0} & \cdots & {\bf 0}\\
  0 &  \cdots & 0 & {\bf 0} & \cdots & {\bf 0}
\end{array}
\right].
\]
\end{enumerate}

Therefore, the set $A=\bigcup_{z\in \mathbb{Z}_{q}^{k(s-2)}}A_{z}$ of size $q^{ks}-q^{k(s-2)}(q^k-1)$ is a $ks$-covering of the NRT
space $\mathbb{Z}_{q}^{2ks}$.
\end{proof}

The construction above is illustrated in the following example.

\begin{example}\label{exe1}
Proposition \ref{triviais} gives the upper bound $K_{2}^{\mathcal{R}}(2,3,3)\leq 8$.
A $3$-covering code of the NRT space $\mathbb{Z}_{2}^{6}$ (NRT poset
$[2\cdot 3]$) that gives the general upper bound for $K_{2}^{\mathcal{R}}(2,3,3)$ is $C=C_{0}\cup C_{1}$, where
\[
C_{0}=\Bigg\{\left[\begin{array}{cc}
          0 & 0 \\
          0 & 0 \\
          0 & 0
        \end{array}
        \right],
        \left[\begin{array}{cc}
          0 & 0 \\
          0 & 1 \\
          0 & 0
        \end{array}
        \right],
        \left[\begin{array}{cc}
          0 & 0 \\
          0 & 0 \\
          0 & 1
        \end{array}
        \right],
        \left[\begin{array}{cc}
          0 & 0 \\
          0 & 1 \\
          0 & 1
        \end{array}
        \right]\Bigg\},
\]
\[
C_{1}=\Bigg\{\left[\begin{array}{cc}
          0 & 1 \\
          0 & 0 \\
          0 & 0
        \end{array}
        \right],
        \left[\begin{array}{cc}
          0 & 1 \\
          0 & 1 \\
          0 & 0
        \end{array}
        \right],
        \left[\begin{array}{cc}
          0 & 1 \\
          0 & 0 \\
          0 & 1
        \end{array}
        \right],
        \left[\begin{array}{cc}
          0 & 1 \\
          0 & 1 \\
          0 & 1
        \end{array}
        \right]\Bigg\}.
\]
Following the construction in Theorem \ref{caseven},  we delete the third codeword of $C_{0}$ and $C_{1}$, and using the remaining codewords, we construct $A_{0}$ and $A_{1}$, respectively.
Let $A=A_{0}\cup A_{1}$, where
\[
A_{0}=\Bigg\{\left[\begin{array}{cc}
          0 & 0 \\
          0 & 0 \\
          0 & 0
          \end{array}
          \right],
          \left[\begin{array}{cc}
          1 & 0 \\
          0 & 1 \\
          0 & 0
          \end{array}
          \right],\left[\begin{array}{cc}
          1 & 0 \\
          1 & 1 \\
          0 & 1
          \end{array}
          \right]\Bigg\}
\ \textrm{and} \
A_{1}=\Bigg\{\left[\begin{array}{cc}
          0 & 1 \\
          0 & 0 \\
          0 & 0
          \end{array}
          \right],
          \left[\begin{array}{cc}
          1 & 1 \\
          0 & 1 \\
          0 & 0
          \end{array}
          \right],
          \left[\begin{array}{cc}
          1 & 1 \\
          1 & 1 \\
          0 & 1
          \end{array}
          \right]\Bigg\}.
\]
Therefore, $A$ is $3$-covering code of the NRT space $\mathbb{Z}_{2}^{6}$ with 6 codewords improving the previous upper bound 8.
We conjecture that $K_{2}^{\mathcal{R}}(2,3,3)=6$.
\end{example}

For $k=1$, Theorem \ref{caseven} improves the bound $q^{s}-q^{s-2}$ given in \cite[Theorem 4]{castoldi2019cai} item $(1)$.

\begin{corollary}\label{casevencai}
For  $q\geq 2$ and $s\geq 3$, $K_{q}^{\mathcal{R}}(2,s,s)\leq q^{s}-q^{s-2}(q-1)$.
\end{corollary}

As a consequence of Theorem \ref{caseven} and \cite[Proposition 17]{castoldi2015covering} we obtain Theorem \ref{bigcaseodd} for $j=0$,  as shown in the next corollary. This result improves upper bounds given in \cite{castoldi2015covering} as displayed in Table \ref{tab3}.

\begin{corollary}\label{caseven0}
For $q\geq 2$ and $s\geq 3$, $K_{q}^{\mathcal{R}}(2k+1,s,(k+1)s)\leq q^{ks}-q^{k(s-2)}(q^{k}-1)$.
\end{corollary}

\subsection{A table comparing new and old upper bounds}

We finish this section by giving a table of upper bounds. We compare the upper bounds on covering codes in NRT spaces obtained in Subsections \ref{sec51} and \ref{sec52} with the upper bounds in the tables given in  \cite{castoldi2015covering}.  The column ``Result'' gives the result in this section used to get the respective bound in the column ``New bounds''. Table \ref{tab3} shows that sometimes the results of this section improve upper bounds from \cite{castoldi2015covering}. These improvements are marked in bold.
\begin{table}[h]
\centering
\caption{Some new upper bounds for $K_{q}^{\mathcal{R}}(m,s,R)$.}\label{tab3}
\begin{tabular}{|c|c|c|c|c|c|c|}
  \hline
  $q$ & $m$ & $s$ & $R$ & Bounds from \cite{castoldi2015covering} & New bounds & Result  \\ \hline
  2 & 2 & 3 & 3 & 8 & {\bf 6} & Corollary \ref{casevencai} \\
  2 & 2 & 4 & 4 & 12 & 12 & Corollary \ref{casevencai}  \\
  2 & 2 & 5 & 5 & 32 & {\bf 24} & Corollary \ref{casevencai}  \\
  2 & 3 & 3 & 3 & 64 & {\bf 48} & Corollary \ref{cases}  \\
  2 & 3 & 3 & 4 & 16 & 24 & Theorem \ref{bigcaseodd}  \\
  2 & 3 & 3 & 5 & 8 & 12 & Corollary \ref{caseodd} \\
  2 & 3 & 3 & 6 & 8 & {\bf 6} & Corollary \ref{caseven0} \\
  2 & 3 & 4 & 7 & 16 & 24 & Corollary \ref{caseodd} \\
  2 & 3 & 4 & 8 & 16 & {\bf 12} & Corollary \ref{caseven0} \\
  2 & 4 & 3 & 6 & 24 & 52 & Theorem  \ref{caseven} \\
  2 & 4 & 4 & 8 & 112 & 208 & Theorem  \ref{caseven} \\
  \hline
\end{tabular}
\end{table}

\section{Constructions of covering codes using ordered covering arrays}\label{sec6}

In this section, ordered covering arrays are used to construct
covering codes in NRT spaces improving upper bounds on their size for larger alphabets. Theorem \ref{teoca}  is a generalization of a result
already shown for covering codes in Hamming spaces connected
with surjective matrices and Maximum Distance Separable (MDS) codes \cite{cohen1997covering}, and its proof has already appeared in the conference paper \cite{castoldi2019cai}.

MDS codes provide a useful tool to construct covering codes in Hamming spaces, see 
\cite{blokhuis1984more,carnielli1985covering,cohen1997covering} for instance. 
MDS codes can be applied in NRT spaces as follows. Suppose that there is an MDS code in the NRT space $\mathbb{Z}_{v}^{ms}$ with minimum distance $d+1$. For every $q\geq 2$, $K_{vq}^{\mathcal{R}}(m,s,d)\leq v^{ms-d}K_{q}^{\mathcal{R}}(m,s,d),$~\cite[Theorem 30]{castoldi2015covering}. An extension of this result based on ordered covering arrays is described below. Indeed, this is analogous to an equivalent result relating to covering codes and covering arrays for the Hamming metric~\cite[Theorem 3.7.10]{cohen1997covering}.

\begin{theorem}\label{teoca} (\cite[Theorem 3]{castoldi2019cai})
Let $v, q, m, s, R$ be positive integers such that \linebreak $0<R<ms$. Then
\[
K_{vq}^{\mathcal{R}}(m,s,R)\leq OCAN(ms-R,m,s,v) K_{q}^{\mathcal{R}}(m,s,R).
\]
\end{theorem}

Theorem \ref{teoca} is very flexible in the sense that if we know upper bounds for $OCAN(ms-R,m,s,v)$ or $K_{q}^{\mathcal{R}}(m,s,R)$, we might improve upper bounds for covering codes in NRT spaces for larger alphabets. 
Indeed, let us illustrate  how the previous results on OCAs and covering codes can be applied in connection with Theorem \ref{teoca}.

Let $m$ be a positive integer such that $(t-1)q+1\leq m \leq (t-1)vq$. Since $K_{q}^{\mathcal{R}}(m,s,ms-t)=q$ (see \cite[Theorem 13]{castoldi2015covering}), Theorem \ref{teoca} implies 
\begin{equation} \label{eqcon1}
K_{vq}^{\mathcal{R}}(m,s,ms-t)\leq OCAN(t,m,s,v)\cdot q.
\end{equation}
In particular, Eq.~(\ref{eqcon1}) when $v=2$ and $t=3$ yields $K_{2q}^{\mathcal{R}}(2m,s,2ms-3)\leq OCAN(3,2m,s,2)\cdot q.$ For $s=2$ or $s=3$, Krikorian  \cite[Theorem 4.2.3]{krikorian2011combinatorial} shows that
\[
OCAN(3,2m,s,2)\leq OCAN(3,m,s,2)+OCAN(2,m,2,2).
\]
By Eq.~(\ref{thm6}),  $OCAN(2,m,2,2)=CAN(2,m,2)$. The combination of the facts above produces 
\[ K_{2q}^{\mathcal{R}}(2m,s,2ms-3) \leq (OCAN(3,m,s,2)+CAN(2,m,2)) \cdot q.
\]

Let $v$ be a prime power, $m\leq v+1$ and $s\leq t$. Since $OCAN(t,m,s,v)=v^{t}$ (Eq. (\ref{eq1})), Theorem \ref{teoca} implies $K_{vq}^{\mathcal{R}}(m,s,ms-t)$ $\leq v^{t} \cdot K_{q}^{\mathcal{R}}(m,s,ms-t).$ \linebreak Moreover, a combination of Eq.~(\ref{corol3}) and  Theorem \ref{teoca} produces \linebreak $K_{(v-1)q}^{\mathcal{R}}(v+1,t,vt)$  $\leq$ $(v^t -2) \cdot K_{q}^{\mathcal{R}}(v+1,t,vt)$.

The discussion above is based on  Corollaries 5 and 6 in \cite{castoldi2019cai}.
Further applications of Theorem \ref{teoca} are available in \cite{castoldi2019cai}.

Here we continue to explore  new relationships in the same spirit as mentioned above. As the main goal, the results from Section \ref{sec5} combined with  Theorem  \ref{teoca} might improve upper bounds for covering codes in NRT spaces.

\begin{corollary}\label{corol50}
Let $v\geq 2$, $q\geq 2$, $s\geq 3$ and $k\geq 1$ be positive integers.
\begin{itemize}
\item[(1)] For $j=0,\ldots,s$, $K_{vq}^{\mathcal{R}}(2k+1,s,(k+1)s-j)\leq$ $OCAN(ks+j,2k+1,s,v)\cdot$ $[q^{ks+j}-q^{k(s-2)+j}(q^{k}-1)]$.
\item[(2)] $K_{vq}^{\mathcal{R}}(2k,s,ks)\leq OCAN(ks,2k,s,v) \cdot [q^{ks}-q^{k(s-2)}(q^{k}-1)]$.
\end{itemize}
\end{corollary}
\begin{proof}
Item $(1)$ follows as an straightforward combination of Theorems  \ref{teoca} and \ref{bigcaseodd}. Analogously, item $(2)$ is derived from Theorems  \ref{teoca} and \ref{caseven}.
\end{proof}

For $k=1$ and a prime power $v$, two consequences of Corollary \ref{corol50} are described in the following result.

\begin{corollary} \label{corol7} Let $v$ be a prime power, $q\geq 2$, and $s\geq 3$. The following bounds hold:
\begin{enumerate}
\item[(1)]  $K_{vq}^{\mathcal{R}}(3,s,2s-1)\leq v^{s+1}\cdot [q^{s+1}-q^{s-1}(q-1)].$
\item[(2)]  $K_{vq}^{\mathcal{R}}(2,s,s)\leq v^s \cdot [q^{s}-q^{s-2}(q-1)].$
\end{enumerate}
\end{corollary}

\begin{proof}
$(1)$ For $v$ a prime power, $OCAN(s+1,v+1,s+1,v)=v^{s+1}$ by Eq.(\ref{eq1}).  Proposition \ref{prop8} item $(3)$ implies that $OCAN(s+1,3,s+1,v)=v^{s+1}$.  This value combined with Corollary \ref{corol50} item $(1)$  (when $k=1$ and $j=1$) derives the desired upper bound.

$(2)$ Since $OCAN(s,v+1,s,v)=v^{s}$ holds by Eq.(\ref{eq1}), Proposition \ref{prop8} item $(3)$ implies that $OCAN(s,2,s,v)=v^s$. An application of Corollary \ref{corol50} item $(2)$  (when $k=1$) concludes the upper bound.
\end{proof}

We now discuss the upper bounds on $K_{vq}^{\mathcal{R}}(3,s,2s-1)$ for a prime power $v$. Corollary \ref{caseodd} gives the upper bound $(vq)^{s+1}-(vq)^{s-1}(vq-1)$. On the other hand, Corollary \ref{corol7} item $(1)$ yields the upper bound  $v^{s+1} \cdot [q^{s+1}-q^{s-1}(q-1)]=(vq)^{s+1}-(vq)^{s-1}v^2(q-1)$.
A closer look reveals that $v^2(q-1)> vq-1$ if and only if $q > \frac{v+1}{v}$. Since $1< \frac{v+1}{v} < 2\leq q$, the inequality $v^2(q-1)> vq-1$ holds.
Therefore, the upper bound from Corollary \ref{corol7} item (1) improves that in  Corollary \ref{caseodd}.

Similarly, let us analyze the upper bounds on $K_{vq}^{\mathcal{R}}(2,s,s)$ for $v$ a prime power. Corollary \ref{casevencai} yields the upper bound  $(vq)^{s}-(vq)^{s-2}(vq-1)$. In contrast, Corollary \ref{corol7} item $(2)$ implies the upper bound $v^s \cdot [q^{s}-q^{s-2}(q-1)]=(vq)^s-(vq)^{s-2}v^2(q-1)$.
By comparing $vq-1$ and $v^2(q-1)$ as we did before, the upper bound from Corollary \ref{corol7} item (2) improves that from  Corollary \ref{casevencai}.

\begin{example} 
We display in Table \ref{tab2} some particular instances showing how much improvement the upper bounds obtained by Corollary \ref{corol7} offer when comparing with Corollaries \ref{caseodd}
and  \ref{casevencai} of Section \ref{sec5}.
Given the parameters $v$ a prime power, $q\geq 2$, $m\in\{2,3\}$, and $s\geq 3$, the penultimate column in  Table \ref{tab2} presents the upper bounds on $K_{vq}^{\mathcal{R}}(m,s,R)$ obtained by Corollary \ref{caseodd}
or  Corollary \ref{casevencai}. The last column in  Table \ref{tab2} indicates the upper bounds derived by  Corollary \ref{corol7} item (1) or item (2). 

\begin{table}[ht]
\centering
\caption{Sample of new upper bounds on $K_{vq}^{\mathcal{R}}(m,s,R)$ for $m=2$ or $m=3$.}
\label{tab2}
\begin{tabular}{|c|c|c|c|c|c|l|l|}
  \hline
  $v$ & $q$ & $v\cdot q$ & $m=3$ & $s$ & $R=2s-1$ & Corollary \ref{caseodd} &  Corollary \ref{corol7} (1)    \\ \hline
  2 & 2 & 4 & 3 & 3 & 5 & 208  & 192   \\
  2 & 2 & 4 & 3 & 4 & 7 & 832  & 768   \\
  2 & 2 & 4 & 3 & 5 & 9 & 3328 & 3072   \\
  2 & 3 & 6 & 3 & 3 & 5 & 1116 & 1008 \\
  2 & 3 & 6 & 3 & 4 & 7 & 6696 & 6048 \\
  2 & 3 & 6 & 3 & 5 & 9 & 40176 & 36288 \\
  3 & 3 & 9 & 3 & 3 & 5 & 5913 & 5103 \\
  3 & 3 & 9 & 3 & 4 & 7 & 53217 & 45927 \\ 
  3 & 3 & 9 & 3 & 5 & 9 & 478953 & 413343 \\
  4 & 2 & 8 & 3 & 3 & 5 & 3648 & 3072 \\ 
  4 & 2 & 8 & 3 & 4 & 7 & 29184 & 24576 \\
  4 & 2 & 8 & 3 & 5 & 9 & 233472 & 196608 \\
  4 & 3 & 12 & 3 & 3 & 5 & 19152 & 16128 \\ 
  4 & 3 & 12 & 3 & 4 & 7 & 229824 & 193536 \\
  4 & 3 & 12 & 3 & 5 & 9 & 2757888 & 2322432 \\
  \hline \hline
  $v$ & $q$ & $v\cdot q$ & $m=2$ & $s$ & $R=s$ & Corollary \ref{casevencai}   & Corollary \ref{corol7} (2)    \\ \hline
  2 & 2 & 4 & 2 & 3 & 3 & 52 & 48   \\
  2 & 2 & 4 & 2 & 4 & 4 & 208 & 192  \\
  2 & 2 & 4 & 2 & 5 & 5 & 832 & 768  \\
  2 & 3 & 6 & 2 & 3 & 3 & 186 & 168  \\
  2 & 3 & 6 & 2 & 4 & 4 & 1116 & 1008  \\
  2 & 3 & 6 & 2 & 5 & 5 & 6696 & 6048  \\
  3 & 3 & 9 & 2 & 3 & 3 & 657 & 567 \\
  3 & 3 & 9 & 2 & 4 & 4 & 5913 & 5103 \\
  3 & 3 & 9 & 2 & 5 & 5 & 53217 & 45927 \\
  4 & 2 & 8 & 2 & 3 & 3 & 456 & 384 \\
  4 & 2 & 8 & 2 & 4 & 4 & 3648 & 3072 \\
  4 & 2 & 8 & 2 & 5 & 5 & 29184 & 24576 \\
  4 & 3 & 12 & 2 & 3 & 3 & 1596 & 1344 \\
  4 & 3 & 12 & 2 & 4 & 4 & 19152 & 16128 \\
  4 & 3 & 12 & 2 & 5 & 5 & 229824 & 193536 \\
\hline
\end{tabular}
\end{table}
\end{example}

\section{Conclusion and further work} 

In this work, we obtain new recursive relations and upper bounds for two different but related combinatorial objects: ordered covering arrays (Section~\ref{sec4}) and covering codes in NRT spaces (Section~\ref{sec5}). In Section~\ref{sec6}, we explore new connections between these two objects and improve several upper bounds for covering codes in NRT spaces as exemplified in Tables \ref{tab3} and \ref{tab2}.

We believe that more results for ordered covering arrays can be explored with different approaches that build on the connections established here.
In addition, results for covering codes can be investigated for any finite poset with analogous connections to appropriate variable strength covering arrays. It is possible that known results on variable strength covering arrays~\cite{raaphorst2013variable,raaphorst2018variable} can be used to improve upper bounds in covering codes for other poset metrics.
There is a lot to explore in this area as much less is known for covering codes  than for minimum distance codes under the poset metric~\cite{firer}.


{\bf Acknowledgments}

A. G. Castoldi was supported by Coordena\c{c}\~{a}o de Aperfei\c{c}amento de Pessoal de N\'{i}vel Superior (CAPES) of Brazil, Science without Borders Program, under Grant Agreement 99999.003758/2014-01.

E. L. Monte Carmelo was partially supported by Conselho Nacional de Desenvolvimento Cient\'{i}fico e Tecnol\'{o}gico (CNPq) - Minist\'{e}rio da Ci\^{e}ncia, Tecnologia e Inova\c{c}\~{a}o (MCTI), under Grant Agreement 311703/2016-0.

L. Moura, D. Panario, and B. Stevens were supported by discovery grants from Natural Sciences and Engineering Research Council of Canada (NSERC).




\end{document}